\documentclass[12pt,a4paper]{amsart}
\usepackage[top=1.15in, bottom=1.15in, left=1.12in, right=1.12in]{geometry}
\usepackage{amsthm, amsmath, amssymb, amscd, latexsym, multicol, verbatim, enumerate, graphics,xypic}
\usepackage{tikz}
\usepackage[colorlinks=true, pdfstartview=FitV, linkcolor=blue, citecolor=blue, urlcolor=blue, breaklinks=true]{hyperref}
\usepackage[english]{babel}

\theoremstyle{definition} 
\newtheorem{coro}{Corollary}
\newtheorem{lem}{Lemma}
\newtheorem{prop}{Proposition}
\newtheorem{thm}{Theorem}
\newtheorem{rem}{Remark}
\newtheorem{defn}{Definition}
\newtheorem*{thm*}{Theorem}
\newtheorem*{lem*}{Lemma}
\newtheorem*{prop*}{Proposition}
\newtheorem*{coro*}{Corollary}

\newtheorem{exam}{Example}
\newtheorem*{exam*}{Example}

\newcommand{\nc}{\newcommand}

\nc{\gl}{\mathfrak{gl}}
\nc{\GL}{\mathfrak{GL}}
\nc{\g}{\mathfrak{g}}
\nc{\gh}{\widehat\g}
\nc{\fh}{\mathfrak{h}}
\nc{\fu}{\mathfrak{u}}
\nc{\la}{\lambda}
\nc{\al}{\alpha }
\nc{\be}{\beta }
\nc{\ve}{\varepsilon }
\nc{\om}{\omega }

\nc{\ta}{\theta}
\nc{\veps}{\varepsilon}
\nc{\ch}{{\mathop {\rm ch}}}
\nc{\Tr}{{\mathop {\rm Tr}\,}}
\nc{\Id}{{\mathop {\rm Id}}}
\nc{\x}{{\bf x}}
\nc{\bs}{{\bf s}}
\nc{\bz}{{\bf z}}
\nc{\bp}{{\bf p}}
\nc{\bfp}{{(\mathbf p)}}
\nc{\bfx}{{\mathbf x}}
\nc{\bfq}{{\bf q}}
\nc{\bq}{{\bf q}}
\nc{\br}{{\bf r}}
\nc{\bc}{{\bf c}}
\nc{\bt}{{\bf t}}
\nc{\bff}{{\bf f}}
\nc{\pa}{\partial}
\nc{\ld}{\ldots}
\nc{\cd}{\cdots}
\nc{\hk}{\hookrightarrow}
\nc{\To}{{\mathbb T}}
\nc{\T}{{\otimes}}
\nc{\gr}{\mathrm{gr}}
\nc{\ov}{\overline}

\nc{\cO}{\mathcal O}
\nc{\msl}{\mathfrak{sl}}
\nc{\msp}{\mathfrak{sp}}
\nc{\mgl}{\mathfrak{gl}}
\nc{\cS}{\mathcal S}
\nc{\U}{\mathrm U}
\nc{\V}{\EuScript V}
\nc{\cal}{\mathcal}
\nc{\bH}{\EuScript H}
\nc{\Res}{\mathrm{Res\ }}
\newcommand{\ol}{\overline}

\newcommand{\bZ}{{\mathbb Z}}

\title[Essential bases and toric degenerations]{Essential bases and toric degenerations arising from birational sequences}
\author{Xin Fang, Ghislain Fourier and Peter Littelmann}
\address{Xin Fang, Mathematisches Institut, Universit\"{a}t zu K\"{o}ln, Weyertal 86-90, D-50931, K\"{o}ln, Germany.}
\email{xinfang.math@gmail.com}
\address{Ghislain Fourier, Institut f\"{u}r Algebra, Zahlentheorie und Diskrete Mathematik, Leibniz Universit\"{a}t Hannover, Germany}
\email{fourier@math.uni-hannover.de}
\address{Peter Littelmann, Mathematisches Institut, Universit\"{a}t zu K\"{o}ln, Weyertal 86-90, D-50931, K\"{o}ln, Germany.}
\email{peter.littelmann@math.uni-koeln.de}

\begin{document}
\maketitle

\begin{abstract}  We present a new approach to construct $T$-equivariant flat toric degenerations of flag varieties and spherical varieties,
combining ideas coming from the theory of Newton-Okounkov bodies with ideas originally stemming from PBW-filtrations.
For each pair $(S,>)$ consisting of a birational sequence and a monomial order, we attach to the affine variety $G/\hskip -3.5pt/U$
a monoid $\Gamma=\Gamma(S,>)$. As a side effect we get a vector space basis $\mathbb B_{\Gamma}$ of $\mathbb C[G/\hskip -3.5pt/U]$,
the elements being indexed by $\Gamma$. The basis $\mathbb B_{\Gamma}$ has multiplicative properties very similar to those of the dual canonical basis.
This makes it possible to transfer the methods of Alexeev and Brion \cite{AB} to this more general setting, once one knows that the
{monoid} $\Gamma$ is finitely generated and saturated.
\end{abstract}

\section{Introduction}
During the recent years, several constructions of $T$-equivariant flat toric degenerations of flag varieties and spherical varieties
have been presented. To name a few examples: the degeneration of $SL_n/B$ by  Gonciulea and Lakshmibai in \cite{GL}; 
its interpretation by Kogan and Miller \cite{KM} using GIT methods; the approach of Caldero \cite{C1} and Alexeev-Brion \cite{AB},
which uses certain nice properties of the multiplicative behavior of the dual canonical basis, there
is the approach of Kaveh \cite{K1} and Kiritchenko \cite{Kir2}, which has been inspired by the theory of Newton-Okounkov bodies \cite{KK}, \cite{LM} (see \cite{And} for constructing toric degenerations using Newton-Okounkov bodies), and there is the approach by Feigin, Fourier and Littelmann in \cite{FFL1}, which has been inspired by a conjecture of Vinberg
concerning certain filtrations of the enveloping algebra of $\mathcal{U}(\mathfrak n^-)$. The starting point for this article was the aim to understand the
connection between \cite{Kir2} and \cite{FFL2} (see Example~\ref{kiritchenkoexample}).
\par
To be more precise, let us fix some notations. In the following let $G$ be a complex connected reductive algebraic group,
we assume that $G\simeq G^{ss}\times (\mathbb C^*)^r$, where the semisimple part $G^{ss}$ of $G$ is simply connected.
We fix a Cartan decomposition $\hbox{\rm Lie\,}G=\mathfrak n^-\oplus\mathfrak h\oplus \mathfrak n^+$, where
$\mathfrak h$ is a maximal torus and $\mathfrak b=\mathfrak h\oplus \mathfrak n^+$ is a Borel subalgebra of $\mathfrak g=\hbox{\rm Lie\,}G$.
Correspondingly let $U^-$ and $U^+$ be the maximal unipotent subgroups of $G$ having $\mathfrak n^-$ respectively
$\mathfrak n^+$ as Lie algebra, and let $T$ be a maximal torus in $G$ with Lie algebra $\mathfrak h$ and weight lattice $\Lambda$.
Denote by $\Lambda^+$ the subset of dominant weights. For $\lambda\in\Lambda^+$, let $V(\lambda)$ be the corresponding finite dimensional irreducible representation and $v_\lambda$ be a highest weight vector. For $\lambda\in\Lambda^+$, let $P_\lambda$ be the stabilizer  of the line $[v_\lambda]\in \mathbb P(V(\lambda))$.
\par
The aim of the present article is to exhibit an approach which, in a certain way, unifies the three approaches mentioned above.
Let $N$ be the number of positive roots of $G$. As a first step, we fix a sequence $S=(\beta_1,\ldots,\beta_N)$ of positive roots
(they do not have to be pairwise different!). We call the sequence a {\it birational sequence} for $U^-$ if the product map
of the associated unipotent root subgroups:
$$
\pi: U_{-\beta_1}\times\cdots \times U_{-\beta_N}\longrightarrow U^-
$$
is birational. A simple example for a birational sequence is the case where $S$ is just an enumeration of the set of all positive roots (the PBW-type case),
another example is the case where $S$ consists only of simple roots and $w_0=s_{\beta_1}\cdots s_{\beta_N}$ is a reduced decomposition of
the longest word in the Weyl group $W$ of $G$. For other examples see section~\ref{examplegenerating}. The second ingredient in
our approach is the choice of a weight function $\Psi:\mathbb N^N\rightarrow\mathbb N$ and a refinement of the associated
partial order ``$>_\Psi$'' to a monomial order ``$>$'' on $\mathbb N^N$.
We use the birational map $\pi$ and the total order ``$>$'' to associate to the variety $G/\hskip -3.5pt/U$ in two different
ways a monoid $\Gamma=\Gamma(S,>)\subset \Lambda \times\mathbb Z^N$. 
\par
Let $(Y,\mathcal{L})$ be a polarized $G$-variety and $R(Y,\mathcal{L})=\bigoplus_{n=0}^\infty H^0(Y,\mathcal{L}^n)$ be the graded algebra. With the help of the birational map $\pi$ and the total order "$>$", we introduce a filtration on $R(Y,\mathcal{L})$ (see section \ref{algfilt}); let $gr\,R(Y,\mathcal{L})$ be the associated graded algebra and $Y_0=\text{Proj}(gr\,R(Y,\mathcal{L}))$ be the corresponding projective variety. If moreover the polarized $G$-variety is spherical, let $P(Y,{\cal L})\subset \Lambda_{\mathbb R}$ and
$P(Y_0,{\cal L})\subset \Lambda_{\mathbb R}\times\mathbb R^{N}$ be the associated moment polytopes.
\par
The main results of the article can be summarized as follows. We fix a birational map $\pi$, a weight function $\Psi$ and a 
monomial order $>$ on $\mathbb N^N$ refining the partial order $>_\Psi$, and let $\Gamma=\Gamma(S,>)\subset \Lambda \times\mathbb Z^N$
be the corresponding monoid associated  to the variety $G/\hskip -3.5pt/U$.
\begin{thm*}\it
If the monoid $\Gamma$ is finitely generated and saturated, then:
\begin{itemize}
\item[{\it i)}] $\Gamma\subset \Lambda\times\mathbb Z^{N}$ is the set of integral points
of a rational polyhedral cone $\mathcal C\subset \Lambda_{\mathbb R}\otimes \mathbb R^N$, and $\hbox{Spec\,}\Gamma$ is a normal toric variety for the
torus $T\times(\mathbb C^*)^N$.
\item[{\it ii)}] Given an affine spherical $G$-variety $Y$, there exists a flat and $T$-equivariant degeneration of $Y$ into a $T\times (\mathbb C^*)^N$-{normal}
toric variety $Y_0$
such that the cone associated to $Y_0$ is the intersection of $\mathcal C$ with $\mathcal W\times \mathbb R^N$, where
$\mathcal W \subset \Lambda_{\mathbb R}$ is the weight cone of $Y$.
\item[{\it iii)}] Given a polarized $G$-variety $(Y,\cal L)$,  there exists a family of $T$-varieties $\psi : {\cal Y}\rightarrow \mathbb A^1$, where ${\cal Y}$ is a
normal variety, together with divisorial sheaves ${\cal O}_{\cal Y}(n)$ ($n\in\mathbb  Z$), such that:
\begin{itemize}
\item $\psi$ is projective and flat.
\item The family $\psi$ is trivial with fiber $Y$ over the complement of $0$ in $\mathbb A^1$, and ${\cal O}_{\cal Y}(n)\vert_Y \simeq {\cal L}^n$
for all $n$.
\item The fiber of $\psi$ at $0$ is isomorphic to $Y_0$, and ${\cal O}_{\cal Y}(n)\vert_{Y_0} \simeq {\cal L}_0^{(n)}$ for all $n$.
\item If $Y$ is spherical, then $Y_0$ is a normal toric variety for the torus $T\times(\mathbb C^*)^N$.
\end{itemize}
\item[{\it iv)}] Given a polarized spherical $G$-variety $(Y,\cal L)$, the projection $p : \Lambda_{\mathbb R}\times \mathbb R^N \rightarrow \Lambda_{\mathbb R}$ onto the first factor
restricts to a surjective map
$p : P(Y_0,{\cal L}) \rightarrow P(Y,{\cal L})$,
with fiber over any $\lambda\in\Lambda^+_{\mathbb R}$
being the essential polytope $P(\lambda)$ (\ref{esspoly}). In particular, for $\lambda\in\Lambda^+$, the limit variety $Y_0$ of the flag variety $Y=G/P_\lambda$ is a normal toric variety under $(\mathbb C^*)^N$, whose
moment polytope is the essential polytope $P(\lambda)$.
\end{itemize}
\end{thm*}
\begin{exam*}
{\it The reduced decomposition case.\/}
For details see section~\ref{stringconeproof}:
fix a reduced decomposition $\underline{w}_0=s_{i_1}\cdots s_{i_N}$ of the longest word $w_0$ in $W$,
let ${C}_{\underline{w}_0}\subset \mathbb R^N$ be the associated string cone defined in \cite{Li1,BZ}
and set $S=({\alpha_{i_1}},\ldots, {\alpha_{i_N}})$. Let $\Psi:\mathbb N^N\rightarrow \mathbb N$ be the height weighted degree
function $\mathbf m\mapsto \sum_{i=1}^N m_i ht(\beta_i)$. We fix on $\mathbb N^N$ the associated $\Psi$-weighted opposite lexicographic order.
The monoid $\Gamma$ is finitely generated and {saturated}, and the cone $\mathcal C$ is the cone
${\mathcal C}_{\underline{w}_0}\subset \Lambda_{\mathbb R}\times {C}_{\underline{w}_0}$ defined by the weight inequalities
in \eqref{weightequation}. So we recover the case studied by Alexeev, Brion \cite{AB} and Kaveh \cite{K1}.
\end{exam*}
\begin{exam*}
{\it Lusztig's PBW-case.\/} For details see section~\ref{Lusztig}: fix a reduced decomposition $\underline{w}_0=s_{i_1}\cdots s_{i_N}$ of the longest word
$w_0$ in $W$ and let $S=(\beta_1,\ldots,\beta_N)$ be the enumeration of the positive roots
associated to the decomposition, i.e., $\beta_k=s_{i_1}\cdots s_{i_{k-1}}(\alpha_k)$ for $k=1,\ldots,N$. Let $\Psi:\mathbb N^N\rightarrow \mathbb N$
be the height weighted degree function as above, we fix on $\mathbb N^N$ the associated  $\Psi$-weighted opposite right
lexicographic order. The monoid $\Gamma$ is finitely generated and {saturated}, the cone $\mathcal C$ can be described using Lusztig's piecewise linear combinatorics
in the {\tt ADE} case respectively the piecewise linear combinatorics developed in \cite{BZ} in the general case.
\end{exam*}
\begin{exam*}
 {\it The homogeneous PBW-case.\/}
For details see section~\ref{homordercase}: let $G=SL_{n+1}$ or $Sp_{2n}$. Let $S=(\beta_1,\dots,\beta_N)$ be a
good ordering of the positive roots and let $\Psi:\mathbb N^N\rightarrow \mathbb N$ be the degree function
$\mathbf m\mapsto \sum_{i=1}^N m_i$. We fix on $\mathbb N^N$ the induced homogeneous right lexicographic order.
The monoid $\Gamma$ is finitely generated and {saturated}, the cone $\mathcal C$
is the cone described in \cite{FFL1}, see also section~\ref{homordercase}.
\end{exam*}
It remains to justify the claim that  the method unifies the three approaches mentioned at the beginning.
Indeed, connections have been observed before: the relation between
Newton-Okounkov bodies and the approach via filtrations has been described in \cite{FFL1} and, vice versa,
Kiritchenko recovers in \cite{Kir2} the polytopes described in \cite{FFL2} as Newton-Okounkov bodies.
The connection between the dual canonical basis, string cones and Newton-Okounkov bodies
is described in \cite{K1}. The aim of this article is to develop for flag varieties (or, more general, for spherical varieties)
an approach such that all three aspects are inherent to the approach from the very beginning.
\par
As a first step we use in section~\ref{SValuationsemigroup} the birational map $\pi$ and the monomial order ``$>$'' to define
a $\Lambda\times \mathbb Z^N$-valued valuation $\nu$ on the field of rational functions $\mathbb C(G/\hskip -3.5pt/U)$
and associate to $G/\hskip -3.5pt/U$ the valuation monoid ${\cal V}(\mathbb C[G/\hskip -3.5pt/U])$. This construction
has been inspired by the theory of Newton-Okounkov bodies.
\par
In the sections \ref{essential1} and \ref{filtandsemi}, we provide a different way to construct the monoid.
The birationality of $\pi$ implies on the level of enveloping algebras that the set of ordered monomials of root vectors
$\{ f_{\beta_1}^{m_1}\cdots f_{\beta_N}^{m_N}\mid \mathbf m=(m_1,\ldots,m_N)\in\mathbb N^N\}$
 forms a spanning set for $\mathcal{U}(\mathfrak n^-)$ as a vector space. In combination  with the fixed total order ``$>$" on $\mathbb N^N$,
we define an increasing filtration of $\mathcal{U}(\mathfrak n^-)$ by finite dimensional subspaces such that the subquotients
are of dimension smaller than or equal to one. As in \cite{FFL1,FFL2,FFL3},
this filtration can be used to associate to $G/\hskip -3.5pt/U$ the monoid $\Gamma$ of {\it essential}
monomials. Essential means in this context that the class of the monomial occurs as
a basis vector of a one dimensional subquotient of the filtration.
\par
We prove in section~\ref{dualessentialsection} that the {\it valuation monoid} $\nu(\mathbb C[G/\hskip -3.5pt/U])$ and the {\it essential monoid} $\Gamma$
coincide. Hence we get a monoid
$\Gamma=\Gamma(S,>)$ attached to $G/\hskip -3.5pt/U$ which depends on the choice of a birational sequence $S$ and the choice of the monomial order ``$>$''.
The most important property proved in this section concerns the third aspect: the basis of $\mathcal{U}(\mathfrak n^-)$ given by the essential monomials
gives via duality rise to a basis of $\mathbb C[G/\hskip -3.5pt/U]$ which we call the {\it dual essential basis}.
It turns out that this vector space basis of $\mathbb C[G/\hskip -3.5pt/U]$ has multiplicative properties similar to those of the dual canonical
basis. So once one knows that $\Gamma$ is finitely generated and {saturated}, then one can view $\hbox{Spec\,}\Gamma$ as a toric variety for the
torus $T\times(\mathbb C^*)^N$, which can be obtained as a flat $T$-equivariant degeneration of $G/\hskip -3.5pt/U$.
One can transfer the approach of Alexeev and Brion \cite{AB} without any change
(up to some details) to this more general setting and prove the same results about $T$-equivariant flat toric degenerations
of spherical varieties and moment polytopes, see sections~\ref{AffG} and \ref{polytopesection}. In section~\ref{Qcf} we discuss the quasi-commutative case, i.e.,
the associated graded vector space $gr\,\mathcal{U}(\mathfrak n^-)$ has in addition an algebra structure
making it isomorphic to $S(\mathfrak n^-)$. If $\Gamma$ is finitely generated and saturated, then the toric variety $Y_0$
inherits in the quasi-commutative case an action of a commutative unipotent group $\mathbb G_a^N$.
\par
As a summarization, the new approach combines ideas coming from the theory of Newton-Okounkov bodies
with ideas originally stemming from PBW-filtrations. As a side effect we get a vector space basis of $\mathbb C[G/\hskip -3.5pt/U]$
which has multiplicative properties very close to those of the dual canonical basis. This makes it possible to
transfer the methods of Alexeev and Brion \cite{AB} to this more general setting, once we know that the
monoid $\Gamma$ is finitely generated and {saturated}.

We conjecture, that $\Gamma$ is always finitely generated. As the Example in \cite{BFF15} concerning the Lie algebra of type $\tt G_2$ shows, $\Gamma$ will not be saturated in general.
But note that the monoid $\Gamma$ and its associated Newton-Okounkov body may provide interesting information even if 
$\Gamma$ is not necessarily finitely generated. For an application where it is not known whether the
monoid is finitely generated or not see \cite{FLP, K2}. In \cite{FLP}, the associated Newton-Okounkov body is used to determine the Gromov width of coadjoint orbits.

\textbf{Acknowledgements.} The work of X.F. is supported by the Alexander von Humboldt Foundation.
G.F. would like to thank the University of Cologne for its hospitality. The work of G.F. and P.L. was partially supported by the DFG-Schwerpunkt SPP 1388.

\section{Glossary}
We summarize the important notations and conventions used throughout the paper in the following glossary:
\begin{itemize}
\item[-] $G$: a connected complex reductive algebraic group with semi-simple part $G^{ss}$;
\item[-] $B$, $T$: a Borel subgroup and a torus contained in the Borel subgroup;
\item[-] $U$, $U^-$: the unipotent radical in a fixed Borel subgroup $B$, and its opposite;
\item[-] $G/\hskip -3.5pt/U:=\text{Spec}(\mathbb{C}[G]^U)$: the $U$-invariant points;
\item[-] $\mathfrak{n}^-$: Lie algebra of $U^-$; $\mathfrak{b}$: Lie algebra of $B$; $\mathfrak{t}$: Lie algebra of $T$;
\item[-] $\mathcal{U}(\mathfrak n^-)$: the universal enveloping algebra of the Lie algebra $\mathfrak{n^-}$;
\item[-] $\Phi$: root system of $\mathfrak{g}$; $\Phi^+$: set of positive roots in $\mathfrak{g}$ with respect to $\mathfrak{b}$; $\mathcal{R}$: root lattice of $\mathfrak{g}$; $\mathcal{R}^+$: semi-group of $\mathcal{R}$ generated by positive roots;
\item[-] $\Lambda$: weight lattice of $G$; $\Lambda^+$: monoid of dominant integral weights; $\Lambda^\dagger$: see section \ref{SValuationsemigroup};
\item[-] $V(\lambda)$: finite dimensional irreducible representation of $G$ of highest weight $\lambda$; $v_\lambda$: a highest weight vector in $V(\lambda)$;
\item[-] $\omega_1,\ldots,\omega_n$: the set of fundamental weights of $G$;
\item[-] $S$: a birational sequence of $G$ (Definition \ref{orderedgenerating});
\item[-] $Z_S$: parametrizing affine space (Definition \ref{Def:ZS}); $\mathcal{Z}_S:=Z_S\times T$;
\item[-] ASM, an affine saturated monoid (Definition \ref{Def:ASM});
\item[-] $es(\mathfrak{n}^-)$: essential monoid of $\mathfrak{n}^-$; $\Gamma(\lambda)$: essential monoid associated to $\lambda$; $\Gamma$: global essential monoid; $\mathcal{C}_{(S,>)}$: the essential cone associated to the birational sequence $S$ and the chosen lexicographic order $>$ (Definition \ref{Def:Ess}).
\end{itemize}

\section{The variety $G/\hskip -4.5pt/U$}\label{GnachU}
\subsection{Some notation} Let $G$ be a connected complex reductive algebraic group
isomorphic to $G^{ss}\times (\mathbb C^*)^r$, where $G^{ss}$ denotes the semisimple part of $G$,
and $G^{ss}$ is simply connected. Let $\hbox{Lie\,}G=\mathfrak g$ be its Lie algebra.
We fix a Cartan decomposition $\mathfrak g=\mathfrak n^-\oplus\mathfrak b$, where
$\mathfrak b$ is a Borel subalgebra with maximal torus $\mathfrak t$ and nilpotent radical
$\mathfrak n^+$.  Let $\Phi$ be the root system of $\mathfrak g$ and denote by $\Phi^+$ the set of
positive roots with respect to the choice of $\mathfrak b$. We denote by ${\cal R}$ the root lattice and by
${\cal R}^+$ the semigroup generated by the positive roots. For $\beta\in \Phi^+$ let
$f_\beta$ be a non-zero root vector in $\mathfrak n^-$ of weight $-\beta$.

Let $B\subset G$ be the Borel subgroup such that $\hbox{Lie\,}B=\mathfrak b$, let $U\subset B$ be its unipotent radical
and denote by $T\subset B$ the maximal torus such that $\hbox{Lie\,}T=\mathfrak t$.
Let $\Lambda$ be the weight lattice of $T$ and denote by $\Lambda^+$ the subset of dominant integral weights.
We write ``$\succ_{wt}$'' for the usual partial order on $\Lambda$, i.e., $\la\succ_{wt}\mu$ if and only if $\la-\mu\in{\cal R}^+$.
We denote by $B^-$ the opposite Borel subgroup and let $U^-$ be its unipotent radical, so $\hbox{Lie\,}U^-=\mathfrak n^-$.
Throughout the paper, $N$ denotes the number of positive roots, $n$ is the rank of $\Phi$ and $n+r$ is the rank of $G$. For $\lambda\in\Lambda^+$, let $V(\lambda)$ be the finite dimensional irreducible representation of $G$ associated to $\lambda$ and $v_\lambda$ be a highest weight vector.

\subsection{The variety $G/\hskip -4pt/U$} One has a natural $G\times G$-action on $G$ by left and right multiplication.
The ring $\mathbb C[G]^{1\times U}$ of $1\times U$-invariant functions (in the following we write just $\mathbb C[G]^{U}$) is
finitely generated and normal, so the variety $G/\hskip -3.5pt/U:=\hbox{Spec}(\mathbb C[G]^U)$ is a normal
(but in general singular) affine variety.
Since $1\times T$ normalizes $1\times U$, $\mathbb C[G]^U$ is a natural $G\times T$-algebra.
As a $G$-representation (resp. $G\times T$-representation) its coordinate ring is isomorphic to
\begin{equation}\label{coordinateGnachU}
\mathbb C[G/\hskip -3.5pt/U]\simeq \bigoplus_{\la\in\Lambda^+} V(\la)^*\simeq \bigoplus_{\la\in\Lambda^+} V(\la)^*\otimes v_\la.
\end{equation}
So  $(1,t)\in G\times T$ acts on $V(\la)^*\simeq V(\la)^*\otimes v_\la$ by the scalar $\la(t)$.

The variety $G/\hskip -3.5pt/U$ is endowed with a natural $G$-action, making it into a spherical variety, i.e., a variety with a dense $B$-orbit.
One has a canonical dominant map $\psi:G\rightarrow G/\hskip -3.5pt/U$, inducing an inclusion $G/U\hookrightarrow G/\hskip -3.5pt/U$
and a birational orbit map  $o:B^-\rightarrow G/\hskip -3.5pt/U$, $b\mapsto b.\bar{1}$ (where $\bar{1}=\psi(1)$).

An element $\phi\in V(\la)^*$ can be seen as a function on the open and dense subset $G/U\hookrightarrow G/\hskip -3.5pt/U$
as follows: for a class $\bar g\in G/U$ let $g\in G$ be a representative. Then
\begin{equation}\label{coordinatefunction}
\phi\vert_{G/U}: G/U\rightarrow \mathbb C,\quad \bar g\mapsto \phi(g v_\la).
\end{equation}

\section{Birational sequences and the variety $Z_S$}
For a root vector $f_{\beta}\in {\mathfrak g}_{-\beta}$,
$\beta\in\Phi^+$, let  $U_{-\beta}:=\{\exp(s f_{\beta})\mid s\in \mathbb C\}$
be the corresponding root subgroup of $G$.
It is well known that the product map $\prod_{\beta\in\Phi^+}U_{-\beta}\rightarrow U^-$
is an isomorphism of affine $T$-varieties for any chosen ordering on the set of positive roots.

\subsection{Birational sequences}
Fix a sequence of positive roots $S=({\beta_1},\ldots,{\beta_N})$.
We make no special assumption on this sequence, for example there may be repetitions, see
Example~\ref{monomialszwo}. Let $\mathbb T$ be the torus $(\mathbb C^*)^N$,
we write $\texttt{t}=(\texttt{t}_1,\ldots,\texttt{t}_N)$ for an element of $\mathbb T$.
\begin{defn}\label{Def:ZS}
The variety  ${Z}_S$ is the affine space $\mathbb A^N$ endowed with the following $T\times\mathbb T$-action:
$$
\forall (t,\texttt{t})\in T\times \mathbb T:(t,\texttt{t})\cdot(z_1,\ldots,z_N):=
(\texttt{t}_1 \beta_1(t)^{-1}z_1,\ldots, \texttt{t}_N \beta_N(t)^{-1}z_N).
$$
\end{defn}
\begin{defn} \label{orderedgenerating}
We call $S$ a {\it birational sequence} for $U^-$ if the product map $\pi$ is birational:
\begin{equation}\label{birational1}
\pi:Z_S\rightarrow U^-,\quad (z_1,\ldots,z_N)\mapsto \exp(z_1f_{\beta_1})\cdots \exp(z_Nf_{\beta_N}).
\end{equation}
\end{defn}

\begin{rem} It would be interesting to know whether in this setup, $\pi$ being dominant implies that $\pi$ is birational.
\end{rem}

We want to analyze the connection between this geometric condition and properties of the enveloping algebra.
For a given sequence $S=({\beta_1},\ldots,{\beta_N})$ of positive roots set
$$
\begin{array}{rl}
\mathcal{U}(\mathfrak n^-)^i&=\langle f_{\beta_i}^{m_i}\cdots f_{\beta_N}^{m_N} \mid \mathbf m=(m_i,\ldots,m_N)\in \mathbb N^{N-i+1}\rangle\subseteq \mathcal{U}(\mathfrak n^-);\\
Y_i&= \{\exp(z_if_{\beta_i})\cdots \exp(z_Nf_{\beta_N})\mid (z_i,\ldots,z_N)\in\mathbb C^{N-i+1}\}\subseteq U^-.
\end{array}
$$
\begin{lem}\label{one}\it Let $\la$ be a dominant weight and let $v_\la\in V(\la)$ be a highest weight vector. The linear span
$\langle Y_iv_\la\rangle= \langle\exp(z_if_{\beta_i})\cdots \exp(z_Nf_{\beta_N})v_\la\mid (z_i,\ldots,z_N)\in\mathbb C^{N-i+1} \rangle$
is equal to $\mathcal{U}(\mathfrak n^-)^i.v_\la$.
\end{lem}
\begin{proof}
Obviously we have  $\langle Y_i.v_\la\rangle\subseteq \mathcal{U}(\mathfrak n^-)^i.v_\la$. We use
an inductive procedure to show equality. Recall that $U_{-\beta_N}(s).v_\la=v_\la+sf_{\beta_N}v_\la+\frac{s^2}{2!}f^2_{\beta_N}v_\la+\ldots$.
The non-zero vectors in $\{v_\la, f_{\beta_N}v_\la,f^2_{\beta_N}v_\la,\ldots\}$ are linearly independent,
so the claim holds for $i=N$.
Assume that the claim holds for all $i>j$, where $j<N$. Note that $Y_{j+1}\subseteq Y_j$
and hence $\langle Y_{j+1}.v_\la\rangle\subseteq\langle Y_j.v_\la\rangle$, so by induction
$\langle Y_j.v_\la\rangle=\langle  U_{-\beta_j}\mathcal{U}(\mathfrak n^-)^{j+1}.v_\la\rangle$. Now
$\mathcal{U}(\mathfrak n^-)^{j+1}.v_\la$ is a $T$-stable subspace
and admits a basis of $T$-eigenvectors. If $v_\mu\in \langle Y_{j+1}.v_\la\rangle$ is a $T$-eigenvector of weight $\mu$,
then
$
U_{-\beta_j}(s).v_\mu=v_\mu+sf_{\beta_j}v_\mu+\frac{s^2}{2!}f^2_{\beta_j}v_\mu+\ldots.
$
Since the non-zero vectors
in the set $\{v_\mu, f_{\beta_j}v_\mu,f^2_{\beta_j}v_\mu,\ldots\}$ are linearly independent, the linear span
$\langle U_{-\beta_j}(s).v_\mu\mid s\in \mathbb C\rangle$ coincides with $\langle v_\mu, f_{\beta_j}v_\mu,  f^2_{\beta_j}v_\mu,\ldots\rangle$,
which proves the claim.
\end{proof}

A first test to show that a given sequence of root vectors is a birational sequence is to
prove that the map $\pi$ in \eqref{birational1} is dominant.
\begin{lem}\it
Let $S=({\beta_1},\ldots,{\beta_N})$ be a sequence of positive roots. The map $\pi$ in \eqref{birational1}
is dominant if and only if the subspaces
$\mathcal{U}(\mathfrak n^-)^i$ form a strictly increasing sequence of subspaces such that
\begin{equation}\label{strict}
\{0\}=\mathcal{U}(\mathfrak n^-)^{N+1}\subsetneq \mathcal{U}(\mathfrak n^-)^N\subsetneq \cdots \subsetneq \mathcal{U}(\mathfrak n^-)^1 =\mathcal{U}(\mathfrak n^-).
\end{equation}
\end{lem}
\begin{proof}
Consider an embedding of $G/B\hookrightarrow \mathbb P(V(\lambda))$ into a projective space
$\mathbb P(V(\lambda))$ for some regular dominant weight $\la$ with sufficiently large coefficients with respect to the fundamental weights, and
identify $U^-$ with the open cell $U^-.id\subset G/B$.
\par
Assume first that \eqref{strict} holds and set $\widehat{(\mathcal{U}(\mathfrak n^-)^{j+1}.v_\la)}=\{v\in  \mathcal{U}(\mathfrak n^-)^{j+1}.v_\la\mid f_{\beta_j}v\not\in
\mathcal{U}(\mathfrak n^-)^{j+1}.v_\la\}$ for $j=1,\ldots,N$. Now \eqref{strict} implies that this set is not empty (for $\la$ as above). The set is open, and it meets
$Y_{j+1}.v_\la$ by Lemma~\ref{one}. It follows that $\dim Y_{j}.v_\la=\dim Y_{j+1}.v_\la+1$, which implies
for $j=1$ that $\pi$ is dominant.

Now assume that the map $\pi$ in \eqref{birational1} is dominant, so for all $\ell\ge 1$ we get
$$
\begin{array}{rcl}
V(\ell\la)\supseteq \langle f_{\beta_1}^{m_1}\cdots f_{\beta_N}^{m_N} v_{\ell\la} \mid\mathbf m\in  \mathbb N^N\rangle
&\supseteq& \langle U_{-\beta_1}(s_1)\cdots U_{-\beta_N}(s_N)
v_{\ell\la}\mid s_1,\ldots,s_N\in\mathbb C\rangle
\\&=&\langle U^- .v_{\ell\la}\rangle  =V(\ell\la).
\end{array}
$$
If one of the inclusions in \eqref{strict} is not strict, say $ \mathcal{U}(\mathfrak n^-)^j= \mathcal{U}(\mathfrak n^-)^{j+1}$, then we would get for all $\ell\ge 1$:
\begin{equation}\label{notstrict}
V(\ell\la)= \langle f_{\beta_1}^{m_1}\cdots  f_{\beta_{j-1}}^{m_{j-1}} f_{\beta_{j+1}}^{m_{j+1}}\cdots f_{\beta_N}^{m_N} v_{\ell\la}
\mid\mathbf m\in  \mathbb N^{N-1}\rangle.
\end{equation}
Let $q_\la$ be the maximal length of a root string in the set of weights of $V(\la)$. It follows that it is sufficient in \eqref{notstrict}
to take only those $\mathbf m\in  \mathbb N^{N-1}$ such that $m_i\le q_\la$ for all $i$.
Note that $q_{\ell\la}=\ell q_\la$. If
\eqref{notstrict} holds, then $\dim V(\ell\la)\le q_\la^{N-1}\ell^{N-1}$, so the dimension is bounded above by a polynomial
in $\ell$ of degree $N-1$. But Weyl's dimension formula implies that $\dim V(\ell\la)$ is a polynomial in $\ell$ of degree $N$,
which is a contradiction, and hence \eqref{strict} holds.
\end{proof}

\subsection{Examples of birational sequences}\label{examplegenerating}
\begin{exam}\label{monomialseins} {\it The PBW-type case\/}:
Fix an enumeration $\Phi^+=\{\beta_1,\ldots,\beta_N\}$ of the positive roots
and set $S=({\beta_1},\ldots,{\beta_N})$.  The map $\pi$ in \eqref{birational1}
above induces an isomorphism of affine varieties for any chosen enumeration. In particular, $S$ is always a birational sequence for $U^-$.
\end{exam}
\begin{exam}\label{monomialszwo} {\it The reduced decomposition case\/}:
Let $\{\alpha_1,\ldots,\alpha_n\}$ be the set of simple roots for the choice of $\Phi^+\subset \Phi$.
Fix a reduced decomposition $\underline{w}_0=s_{i_1}\cdots s_{i_N}$ of the longest word in the Weyl group and
set $S=({\alpha_{i_1}},\ldots, {\alpha_{i_N}})$. Using the  Bott-Samelson desingularization associated to the reduced
decomposition of $w_0$, one easily shows that $S$ is a birational sequence.
\end{exam}
\begin{exam} {\it Mixed cases\/}:  \rm Let $L\subset G$ be a Levi subgroup and denote by $w_L\in W_L$ the
longest element in the Weyl group $W_L\subset W$ of $L$. Let $(\alpha_{i_1},\ldots,\alpha_{i_t})$ be a sequence
of simple roots for $G$ such that $s_{i_1}\cdots s_{i_t} w_L=w_0$ and $t+\ell(w_L)=N$. Denote by $\Phi_L$ the root
system of $L$ and let $\{\beta_1,\ldots,\beta_{N-t}\}=\Phi^+\cap \Phi_L$, then $S=(\alpha_{i_1},\ldots,\alpha_{i_t},\beta_1,\ldots,\beta_{N-t})$
is a birational sequence.

Let $w_L=s_{\gamma_1}\cdots s_{\gamma_{N-t}}$
be a reduced decomposition and let $\{\delta_1,\ldots,\delta_{t}\}$ be the roots in $\Phi^+-\Phi_L^+$,
then $S=(\delta_{i_1},\ldots,\delta_{i_t},\gamma_1,\ldots,\gamma_{N-t})$ is a birational sequence.
\end{exam}
\begin{exam} {\it More mixed cases\/}:  Of course, one may repeat the procedure above: Given
an increasing sequence of Levi subgroups $L_1\subsetneq L_2\subsetneq\ldots  \subsetneq L_r=G$,
one subsequently switches between Example~\ref{monomialseins} and Example~\ref{monomialszwo} . For example,
for $G=SL_5$ let $\alpha_i=\epsilon_i-\epsilon_{i+1}$, $i=1,\ldots,4$, and let $L_i$ be the Levi
subgroup associated to the simple roots $\alpha_1,\ldots,\alpha_i$. One gets the following birational sequence:
$$
S=(
\underbrace{
\epsilon_4-\epsilon_{5},\epsilon_3-\epsilon_{5},\epsilon_2-\epsilon_{5},\epsilon_1-\epsilon_{5}}_{Example~\ref{monomialseins}},
\underbrace{
\epsilon_1-\epsilon_{2},\epsilon_2-\epsilon_{3},\epsilon_3-\epsilon_{4}}_{Example~\ref{monomialszwo}},
\underbrace{
\epsilon_2-\epsilon_{3},\epsilon_1-\epsilon_{3}}_{Example~\ref{monomialseins}},
\epsilon_1-\epsilon_{2})
$$
\end{exam}
\begin{exam} {\it Elementary moves\/}:
Let  $S=({\beta_1},\ldots,{\beta_N})$ be a birational sequence for $U^-$. If the root vectors
$f_{\beta_i},f_{\beta_{i+1}}$ commute, then switching $\beta_i$ and $\beta_{i+1}$ gives again
a birational sequence for $U^-$.

If $\{\beta_i,\beta_{i+1},\beta_{i+2}\}$
is the set of positive roots for a (sub-) root system of type ${\tt A}_2$, then permuting the three gives again
a birational sequence for $U^-$. If $\beta_i=\beta_{i+1}+\beta_{i+2}$, then replacing $\beta_i,\beta_{i+1},\beta_{i+2}$
by $\beta_{i+2},\beta_{i+1},\beta_{i+2}$ or $\beta_{i+1},\beta_{i+2},\beta_{i+1}$ gives again a birational sequence for $U^-$
and so on. Similar rules hold of course for other rank two (sub-) root systems, too.
\end{exam}
\begin{exam} If $S=({\beta_1},\ldots,{\beta_N})$ is a birational sequence, then so is $S'=({\beta_N},\ldots,{\beta_1})$.
\end{exam}

\section{A $T$-equivariant birational map}
Let $S=({\beta_1},\ldots,{\beta_N})$ be a birational sequence for $U^-$.
Let ${\cal Z}_S$ be the toric variety ${\cal Z}_S= Z_S\times T$, where the torus
$T\times \mathbb T$ is acting on ${\cal Z}_S$ as follows:
$$
\forall (t,\texttt{t})\in T\times \mathbb T:(t,\texttt{t})\cdot(z,t'):=
(\texttt{t}_1 \beta_1(t)^{-1}z_1,\ldots, \texttt{t}_N \beta_N(t)^{-1}z_N;tt').
$$
Let $x_i:\, (z_1 ,\ldots,z_N)\mapsto z_i$ be the $i$-th coordinate function on ${Z}_S$,
then
$$
\mathbb C[{\cal Z}_S]\simeq \mathbb C[x_1,\ldots,x_N]\otimes \mathbb C[T]=\mathbb C[x_1,\ldots,x_N]\otimes\mathbb C[e^{\la}\mid\la\in\Lambda].
$$
\begin{lem}\label{Lem:3}
\it Let $S=(\beta_1,\ldots,\beta_N)$ be a birational sequence.
The canonical map ${\cal Z}_S\rightarrow G/ U$, $(z,t)\mapsto
\pi(z).t\cdot \bar{1}$, induces a $T$-equivariant birational map $\varphi: {\cal Z}_S\rightarrow G/\hskip -3.5pt / U$.
\end{lem}
\begin{proof}
The map  is birational because the maps $\pi\times id:Z_S\times T\rightarrow  U^-\times T$,
the map $U^-\times T\rightarrow B^-$ and the orbit map $o: B^-\rightarrow G/\hskip-3.5pt/ U$
are birational. For the actions of $T$ on ${\cal Z}_S$ and $G/\hskip-3.5pt/ U$ we have
$$
\begin{array}{rclcl}
\varphi((t\times \texttt{1})\cdot (z_1,\ldots,z_N,t'))&=&
\exp(\frac{z_1}{\beta_1(t)}f_1)\cdots \exp(\frac{z_N}{\beta_N(t)}f_N)tt' \cdot\bar{1}\\
&=&(t\exp(z_1f_1)t^{-1})\cdots (t\exp(z_Nf_N)t^{-1})tt' \cdot\bar{1}\\
&=&t\cdot\varphi(z_1,\ldots,z_N,t').\\
\end{array}
$$
\end{proof}
\begin{coro}\label{embedding}\it
Via the isomorphism $\mathbb C(G/\hskip -3.5pt/U)\simeq\mathbb C({\cal Z}_S)$,
we can identify the coordinate ring $\mathbb C[G/\hskip -3.5pt/U]$ with a subalgebra of
$A:=\mathbb C[x_1,\ldots,x_N]\otimes\mathbb C[e^{\la}\mid\la\in\Lambda^+]\subset \mathbb C({\cal Z}_S)$.
\end{coro}

\begin{rem}
The birational map $\varphi:\mathcal{Z}_S\rightarrow G/\hskip -3.5pt/U$ in Lemma \ref{Lem:3} can be extended to a morphism $\widetilde{\varphi}:\mathbb{A}^{N+n}\rightarrow G/\hskip -3.5pt/U$ in the following way. We identify $G/U$ with the dense orbit 
$G\cdot (v_{\omega_1}+\ldots+v_{\omega_n})$ in $G/\hskip -3.5pt/U$ (see \cite{VP}), where $v_{\omega_k}$ is a highest 
weight vector in the fundamental representation $V(\omega_k)$ of $G$ and $v_{\omega_1}+\ldots+v_{\omega_n}\in V(\omega_1)\oplus\cdots\oplus V(\omega_n)$. The boundary of this dense orbit is the union of the $G$-orbits $G\cdot (v_{\omega_{j_1}}+\ldots+v_{\omega_{j_p}})\subset V(\omega_1)\oplus\cdots\oplus V(\omega_n)$ where $1\leq j_1<\ldots<j_p\leq n$ and $0\leq p\leq n-1$ (when $p=0$, this orbit is a single point $\{0\}$).
\par
We fix the embedding of $T$ into $\mathbb{A}^n$: 
$$t\in T\mapsto (\omega_1(t),\omega_2(t),\ldots,\omega_n(t))\in\mathbb{A}^n,$$
identifying $T$ with points in $\mathbb{A}^n$ having non-zero coordinates. With these notations, the morphism 
$\varphi:\mathcal{Z}_S\rightarrow G/U$ is given by
$$(z,t)\mapsto \pi(z)\cdot (\omega_1(t)v_{\omega_1}+\ldots+\omega_n(t)v_{\omega_n}),$$
which can be extended to a morphism $\widetilde{\varphi}:\mathbb{A}^{N+n}\rightarrow G/\hskip -3.5pt/U$, by the above argument.
\end{rem}

\section{Weighted lex-orders on $\mathbb N^N$}
We want to obtain a $T$-equivariant flat degeneration of $G/\hskip -3.5pt/U$ into a toric variety using the embedding in Corollary~\ref{embedding}.
To achieve this, we replace $\mathbb C[G/\hskip -3.5pt/U]$ by its ``{\it algebra of initial terms\/}''.
We recall the definition and some properties of monomial orders used in the following.

\subsection{Weight functions and lexicographic orders} Let $\Psi: \mathbb Z^N\rightarrow \mathbb Z$
be a $\mathbb Z$-linear map such that $\Psi(\mathbb N^N)\subseteq \mathbb N$, we call $\Psi$
an  {\it integral weight function}. The {\it weight order} on $\mathbb N^N$ associated to $\Psi$ is the partial
order defined by $\mathbf m>_{\Psi} \mathbf m'$ iff $\Psi(\mathbf m)>\Psi(\mathbf m')$.
We write $>_{lex}$ for the lexicographic order and $>_{rlex}$
for the right lexicographic order on $\mathbb N^N$.
\begin{exam}
For $\mathbf a=(a_1,a_2,a_3), \mathbf b=(b_1,b_2,b_3)\in\mathbb N^3$
we have $\mathbf a>_{lex} \mathbf b$ if $a_1>b_1$, or if $a_1=b_1$ and $a_2>b_2$, or if $a_1=b_1$, $a_2=b_2$  and $a_3>b_3$,
and we have $\mathbf a>_{rlex} \mathbf b$ if $a_3>b_3$, or if $a_3=b_3$ and $a_2>b_2$, or if $a_3=b_3$, $a_2=b_2$  and $a_1>b_1$.
\end{exam}
We refine the partial order above to a total order.
\begin{defn}\label{weightorder}
A {\it $\Psi$-weighted lexicographic order} (respectively a {\it $\Psi$-weighted right lexicographic order}) on
$\mathbb N^N$ is a total order ``$>$" on $\mathbb N^N$ refining ``$>_{\Psi}$'' as follows:
$$\Small
\mathbf m>\mathbf m'\Leftrightarrow \hbox{\ either $\Psi(\mathbf m)> \Psi(\mathbf m')$, or
$\Psi(\mathbf m)=\Psi(\mathbf m')$ and $\mathbf m>_{lex}\mathbf m'$ (respectively $\mathbf m>_{rlex}\mathbf m'$)}.
$$
If $\Psi$ satisfies in addition the condition $\Psi(\mathbf m)>0$ for all $\mathbf m\in \mathbb N^N-\{0\}$,
then a {\it $\Psi$-weighted opposite lexicographic order} (respectively {\it $\Psi$-weighted opposite right lexicographic order}) on
$\mathbb N^N$ is a total order ``$>$" on $\mathbb N^N$ defined as follows:
$$\Small
\mathbf m>\mathbf m'\Leftrightarrow \hbox{\ either $\Psi(\mathbf m)> \Psi(\mathbf m')$, or
$\Psi(\mathbf m)=\Psi(\mathbf m')$ and $\mathbf m<_{lex}\mathbf m'$ (respectively $\mathbf m<_{rlex}\mathbf m'$)}.
$$
\end{defn}
\noindent
{\bf Notation:} We often just write $(\mathbb N^N,>)$ and say that $\mathbb N^N$ is endowed with a {\it $\Psi$-weighted lex order}
if we have an integral weight function $\Psi:\mathbb Z^N\rightarrow \mathbb Z$
and ``$>$''  is one of the refinements of ``$>_\Psi$'' above.

It is easy to verify:
\begin{lem} \it
A $\Psi$-weighted lex order defines a monomial order on $\mathbb N^N$, i.e., a total order such that $\bf m>m'$
implies for all ${\bf m''}\in \mathbb N^N-\{0\}$: $\bf m+m''>m'+m''>m'$.
\end{lem}
\begin{exam}\label{lexorder}
If $\Psi$ is the zero map, then the $\Psi$-weighted lexicographic order is just the lexicographic order,
the $\Psi$-weighted right lexicographic order is just the right lexicographic order on $\mathbb N^N$.
\end{exam}
\begin{exam}\label{homorder}
If $\Psi:\mathbb Z^N\rightarrow \mathbb Z$ is the map $\mathbf m\mapsto \sum_{i=1}^N m_i$, then the $\Psi$-weighted lexicographic order
is the homogeneous lexicographic order, and the $\Psi$-weighted right lexicographic order is the
homogeneous right lexicographic order on $\mathbb N^N$.
\end{exam}
\begin{exam}\label{rootorder}
Fix a sequence $S=(\beta_1,\ldots,\beta_N)$ of roots in $\Phi^+$, let $ht$ be the height function on the positive roots and let
$\Psi$ be the {\it height weighted function}:
$$
\Psi:\mathbb Z^N\rightarrow \mathbb Z,\quad \mathbf m\mapsto \sum_{i=1,\ldots,N} m_i ht(\beta_i).
$$
Now $\forall\mathbf m\in \mathbb N^N-\{0\}: \Psi(\mathbf m)>0$, so one can define
all four possible $\Psi$-weighted lex orders on $\mathbb N^N$.
\end{exam}

\section{The valuation monoid}\label{SValuationsemigroup}
Having fixed a $\Psi$-weighted lex order ``$>$'' on $\mathbb N^N$, we define a $\mathbb Z^N$-valued valuation on $\mathbb C[x_1,\ldots,x_N]-\{0\}$ by
\begin{equation}\label{valuationdef1}
\nu_1(p(\bfx))=\min\{\mathbf p \in \mathbb N^N\mid a_{\mathbf p}\not=0\}\quad \hbox{for\ }p({\bf x})
=\sum_{\mathbf p \in \mathbb N^N} a_{\bf p} {\bf x}^{\bf p}.
\end{equation}
For a rational function $h = \frac{p}{p'}\in \mathbb C(x_1,\ldots,x_N)$ we define $\nu_1(h)=\nu_1(p)-\nu_1(p')$.
The valuation $\nu_1$ is usually called the {\it lowest term valuation}.
To extend the valuation to $\mathbb C(G/\hskip-3.5pt/U)$, fix characters $\eta_1,\ldots,\eta_r$ such
that
$$
X(\mathbb C^*)^r=\mathbb Z\eta_1\oplus\ldots\oplus \mathbb Z\eta_r
$$
Set $\Lambda^\dagger=\mathbb N\omega_1\oplus\ldots \oplus \mathbb N\omega_n\oplus
\mathbb N\eta_1\oplus\ldots \oplus \mathbb N\eta_r\subset \Lambda^+$
and fix a total order $>$ on $\Lambda^\dagger\simeq\mathbb N^{n+r}$. We define a total
order ``$\triangleright$'' on $\Lambda^\dagger \times \mathbb N^N$ by: $(\la,\mathbf m)\triangleright(\mu,\mathbf m')$
if  $\la > \mu$, and if $\la = \mu$, then we set  $(\la,\mathbf m)\triangleright(\la,\mathbf m')$
if $\mathbf m>\mathbf m'$.
Given  $p(\mathbf x,\mathbf e^{\la})\in \mathbb C[x_1,\ldots,x_N; e^{\la}: \la\in \Lambda^\dagger]-\{0\}$,
we define a $\Lambda\times\mathbb Z^N$-valued valuation:
\begin{equation}\label{valuationdef2}
\nu(p(\bfx, e^{\la}))=\min\{(\la,\mathbf p) \in  \Lambda^\dagger\times\mathbb N^N\mid a_{\la,\mathbf p}\not=0\} \hbox{\ for\ }
p(\mathbf x, e^{\la}) = \sum_{(\la, {\bf p})\in  \Lambda^\dagger\times\mathbb N^N} a_{\la,\bf p} {\bf x}^{\bf p} e^\la,
\end{equation}
and for a rational function $h = \frac{p}{p'}$ we define $\nu(h)=\nu(p)-\nu(p')$.
\begin{defn}
The valuation $\nu$ is called the {\it lowest term valuation} with respect to the
para\-meters $x_{{1}},\ldots, x_{N}, e^{\omega_1},\ldots,e^{\omega_n},e^{\eta_1},\ldots,e^{\eta_r}$ and the monomial order ``$\,\ge$''.
\end{defn}
By Corollary~\ref{embedding}, it makes sense to view $\nu$ also as a valuation on $\mathbb C(G/\hskip -3.5pt/U)$.
We associate to $\mathbb C[G/\hskip -3.5pt/U]$ the valuation monoid ${\cal V}(G/\hskip -3.5pt/U)$:
\begin{equation}\label{valuationsemigroup}
{\cal V}(G/\hskip -3.5pt/U)={\cal V}(G/\hskip -3.5pt/U,\nu,>)
=\{\nu({p})\mid p\in \mathbb C[G/\hskip -3.5pt/U]-\{0\}\}
\subseteq \Lambda\times\bZ^N.
\end{equation}
\begin{rem}\label{nueinsbewertung}
The monoid ${\cal V}(G/\hskip -3.5pt/U)$ is independent of the choice of the total order on $\Lambda^\dagger$.
For an element $\phi\in V(\la)^*\hookrightarrow \mathbb C[G/\hskip -3.5pt/U]$ one has (see \eqref{coordinatefunction})
$\nu(\phi)=\nu(e^\la p(\mathbf x))=(\la,\nu_1(p(\mathbf x)))$. So \eqref{coordinateGnachU} implies
that ${\cal V}(G/\hskip -3.5pt/U)$ depends only on the choice of the $\Psi$-weighted lex
order ``$>$''.
\end{rem}
\begin{rem}
For an element $\phi\in V(\la)^*\hookrightarrow \mathbb C[G/\hskip -3.5pt/U]$
the valuation $\nu(\phi)=(\la,\mathbf p)$ measures the ``vanishing behavior'' of $\varphi^*(\phi)\vert_{Z_S\times \{1\}}$ near $0\in Z_S$.
A serious geometric interpretation depends of course on the choice of the monomial order. If ``$>$'' is for
example the lexicographic order, then $\nu(\phi)=(\la,\mathbf p)$ implies that $\varphi^*(\phi)\vert_{Z_S\times \{1\}}$ vanishes with multiplicity $p_1$
on the hyperplane
$\{x_1\equiv 0\}\subset Z_S$ , the restriction $x_1^{-p_1}\varphi^*(\phi)\vert_{\{x_1\equiv 0\}\subset Z_S\times \{1\}}$
vanishes on the intersection $\{x_1\equiv 0\}\cap \{x_2\equiv 0\}$ with multiplicity $p_2$ etc.
If $\pi$ is an isomorphism in a neighborhood of the class of the identity $\bar{1\hskip-3pt \hbox{I}}$ in $G/\hskip-3.5pt /U$,
which is true for example in the PBW-type case, the valuation $\nu(\phi)=(\la,\mathbf p)$ measures in this sense the ``vanishing behavior'' of
$\phi$ near $\bar{1\hskip-3pt \hbox{I}}\in U^-$.
\end{rem}
\section{Filtrations and essential elements for $\mathcal{U}(\mathfrak n^-)$}\label{essential1}
We present a different description of the valuation monoid.
Given $f\in \mathfrak n^-$, we write $f^{(m)}$ for the divided power $\frac{f^m}{m!}$ in $\mathcal{U}(\mathfrak n^-)$.
Fix a birational sequence $S=(\beta_1,\ldots, \beta_N)$ for $U^-$. For $\mathbf m\in \mathbb N^N$,
we write $\mathbf f^{(\mathbf m)}=f_{\beta_1}^{(m_1)}\cdots f_{\beta_N}^{(m_N)}$ for the ordered product of the divided powers of the elements
and we call $wt(\mathbf m)=m_1\beta_1+\ldots+m_N\beta_N$ the {\it weight of
$\mathbf m$}. Recall that $\mathbf f^{(\mathbf m)}$ is an eigenvector of weight $-wt(\mathbf m)$ for the adjoint action of $T$ on $\mathcal{U}(\mathfrak n^-)$.

\subsection{Filtrations and essential elements} Let $\mathbb N^N$ be endowed with a $\Psi$-weighted lex order ``$>$''.
\begin{defn} Let $S=(\beta_1,\ldots, \beta_N)$ be a birational sequence for $U^-$.
A $(\mathbb N^N,>,S)$-{\it filtration} of $\mathcal{U}(\mathfrak n^-)$ is an increasing sequence of subspaces $\mathcal{U}(\mathfrak n^-)_{\bf \le m}$,
$\mathbf m\in \mathbb N^N$, defined by
$$
\mathcal{U}(\mathfrak n^-)_{\bf \le m}=\langle\mathbf f^{(\mathbf k)}=f_{\beta_1}^{(k_1)}\cdots f_{\beta_N}^{(k_N)}\mid \mathbf k\le \mathbf m \rangle.
$$
The associated graded vector space is denoted by $\mathcal{U}^{gr}(\mathfrak n^-)=
\bigoplus_{{\bf m}\in \mathbb N^N} \mathcal{U}^{gr}(\mathfrak n^-)_{\bf m}$, where
$$
\mathcal{U}^{gr}(\mathfrak n^-)_{\bf m}:=\mathcal{U}(\mathfrak n^-)_{\bf \le m}/\mathcal{U}(\mathfrak n^-)_{\bf < m}\hbox{\ and\ }
\mathcal{U}(\mathfrak n^-)_{\bf < m}=\langle \mathbf f^{(\mathbf k)}\mid \mathbf k< \mathbf m \rangle.
$$
\end{defn}

\begin{defn}
An element $\mathbf m\in \mathbb N^N$ is called {\it essential} for the $(\mathbb N^N,>,S)$-filtration if
$\mathcal{U}^{gr}(\mathfrak n^-)_{\bf m}\not=0$.
Denote by $es(\mathfrak n^-)\subseteq  \mathbb N^N$ the set of all {\it essential multi-exponents}.
\end{defn}
\begin{exam}\label{monomialseinsa}
Fix an enumeration $\Phi^+=\{\beta_1,\ldots,\beta_N\}$ of the positive roots
and set $S=({\beta_1},\ldots, {\beta_N})$. By Example~\ref{monomialseins} we know that
$S$ is a birational sequence for $U^-$. So for any choice of a $\Psi$-weighted lex order on $\mathbb N^N$
we get a $(\mathbb N^N,>,S)$-{\it filtration} of $\mathcal{U}(\mathfrak n^-)$ and an associated graded
space  $\mathcal{U}^{gr}(\mathfrak n^-)$. The ordered monomials form a PBW-basis for $\mathcal{U}(\mathfrak n^-)$, so
all ordered monomials are essential, independently of the choice of ``$>$''. It follows:
$$
es(\mathfrak n^-)=\mathbb N^N.
$$
\end{exam}
\begin{exam}\label{sl3}
If $\mathfrak g=\mathfrak{sl}_3$ and $w_0=s_1s_2s_1$, then $S=(f_{\alpha_1},f_{\alpha_2},f_{\alpha_1})$ is a
birational sequence by Example~\ref{monomialszwo}. Let ``$\ge$"
be the right lexicographic ordering on $\mathbb N^3$.
The Serre relation $f_{\alpha_1}^{(2)}f_{\alpha_2}- f_{\alpha_1}f_{\alpha_2}f_{\alpha_1}+f_{\alpha_2}f_{\alpha_1}^{(2)}=0$
implies $\mathcal{U}(\mathfrak n^-)_{\le (0,1,2)}/ \mathcal{U}(\mathfrak n^-)_{<(0,1,2)}=0$, so $(0,1,2)\not\in es(\mathfrak n^-)$.
\end{exam}
\begin{exam}\label{GZPattern}
Consider in Example~\ref{monomialszwo} the Lie algebra $\mathfrak{sl}_{n+1}$ and the reduced decomposition
$w_0=s_1(s_2s_1)(s_3s_2s_1)\cdots(s_n s_{n-1}\cdots s_2 s_1)$. Set
$$
S=({\alpha_{1}},{\alpha_{2}},{\alpha_{1}},{\alpha_{3}},{\alpha_{2}},{\alpha_{1}},\ldots,{\alpha_{n}},{\alpha_{n-1}},
{\alpha_{n-2}},\ldots, {\alpha_{2}},{\alpha_{1}}),
$$
so $S$ is a birational sequence. Fix on $\mathbb N^N$ the right lexicographic order. Then \cite{Li2} (Theorem 17, 18) implies:
$$
es(\mathfrak n^-)=\bigg\{(p_{1,1}, p_{2,2},p_{2,1},\ldots, p_{n,n},p_{n,n-1},\ldots, p_{n,1})\in \mathbb N^N\mid
\substack{p_{2,2}\ge p_{2,1}, \\ p_{3,3}\ge p_{3,2}\ge p_{3,1},\\ \ldots\\ p_{n,n}\ge p_{n,n-1}\ge \ldots \ge p_{n,2}\ge  p_{n,1}.}\bigg\}
$$
\end{exam}
\begin{exam}\label{string1}
Fix a reduced decomposition $w_0=s_{i_1}\cdots s_{i_N}$ of the longest word in the Weyl group of $\mathfrak g$,
we write $\underline{w}_0$ to indicate that we consider $w_0$ together with a fixed decomposition.
Set $S=({\alpha_{i_1}},\ldots, {\alpha_{i_N}})$, then $S$ is a birational sequence by Example~\ref{monomialszwo}.
We use on $\mathbb N^N$ the $\Psi$-weighted opposite lexicographic order defined in Example~\ref{rootorder},
for details see section~\ref{stringconeproof}. Let ${C}_{\underline{w}_0}\subset \mathbb R^N$ be the string cone defined in \cite{BZ, Li1}.
We will see in section~\ref{stringconeproof}, Theorem~\ref{stringCone}:
$$
es(\mathfrak n^-)={C}_{\underline{w}_0}\cap \mathbb Z^N.
$$
\end{exam}

\section{Induced filtrations and essential monoids}\label{filtandsemi}
A filtration on $\mathcal{U}(\mathfrak n^-)$ induces a filtration on a highest weight representation of $G$. We extend the notion
of an essential element to irreducible $G$-representations and show that the set of essential tuples is naturally endowed with the structure of a monoid.
\subsection{Induced filtrations and essential elements}
For a dominant weight $\la\in\Lambda^+$ let $V(\la)$ be the irreducible $G$-representation
of highest weight $\la$ and fix a highest weight vector $v_\la$. Given a $(\mathbb N^N,>,S)$-filtration on $\mathcal{U}(\mathfrak n^-)$,
we get an induced $(\mathbb N^N,>,S)$-filtration on $V(\la)$ as follows:
$$
V(\la)_{\bf \le m}:=\mathcal{U}(\mathfrak n^-)_{\bf \le m}  v_\la, \quad V(\la)_{\bf < m}:=\mathcal{U}(\mathfrak n^-)_{\bf <m}  v_\la,\quad \forall m\in \mathbb N^N.
$$
We set
$$
V^{gr}(\la)=\bigoplus_{{\bf m}\in \mathbb N^N} V(\la)_{\bf m},\hbox{\  where\ }V(\la)_{\bf m}= V(\la)_{\bf \le m}/V(\la)_{\bf <m}.
$$
Obviously we have  $\dim V(\la)_{\bf m}\le 1$.
\begin{defn}
A pair  $(\la,{\bf m})\in \Lambda^+\times \mathbb N^N$ is called {\it essential} for the $(\mathbb N^N,>,S)$-filtration of $V(\la)$ if $V(\la)_{\bf m}\not=0$.
If $(\la,\bf m)$ is essential, then $\mathbf f^{(\mathbf m)}v_\la$ is called an {\it essential vector} for $V(\la)$ and
$\bf m$ is called an {\it essential multi-exponent} for $V(\la)$.
Let  $es(\la)\subseteq \mathbb N^N$ be the set of all essential multi-exponents for $V(\la)$.
\end{defn}
The term {\it essential} makes sense only in connection with a fixed $(\mathbb N^N,>,S)$-filtration on $\mathcal{U}(\mathfrak n^-)$.
We often omit the reference to a filtration and assume tacitly that we have fixed one.
\begin{rem}
$\{\mathbf f^{(\mathbf m)}\mid \mathbf m\in es(\mathfrak n^-)\}\subset \mathcal{U}(\mathfrak n^-)$
and $\{\mathbf f^{(\mathbf m)}v_\la \mid \mathbf m\in es(\la)\}\subset V(\la)$ are bases consisting
of $T$-eigenvectors.
\end{rem}
We write $\lambda\gg 0$ if $\lambda$ is regular and has sufficiently large coefficients with respect to the fundamental weights. 
\begin{lem}\label{repessentialversusessential}\it
\begin{itemize}
\item[{\it i)}] If $(\la,\mathbf m)$ is essential for $V(\la)$, then $\mathbf m$ is essential  for $\mathcal{U}(\mathfrak n^-)$, so
$es(\la)\subseteq es(\mathfrak n^-)$.
\item[{\it ii)}] For a given $\la\in\Lambda^+$ there exists only a finite number of $\mathbf m\in es(\mathfrak n^-)$
such that $\mathbf f^{(\mathbf m)}v_\la\not=0$.
\item[{\it iii)}] If $\mathbf m$ is essential  for $\mathcal{U}(\mathfrak n^-)$, then $(\la,\mathbf m)$ is essential for $V(\la)$ for
$\lambda\gg 0$.
\end{itemize}
\end{lem}
\begin{proof}
Part {\it i)} is obvious, {\it ii)} follows immediately by weight arguments. Let $\mathcal{U}(\mathfrak n^-)_{-wt(\mathbf m)}$ be the
$T$-weight space with respect to the adjoint action.
For all  $\la\gg 0$ one knows that the
map $\mathcal{U}(\mathfrak n^-)_{-wt(\mathbf m)}\hookrightarrow V(\la)$, $u\mapsto u v_\la$, is injective, which proves {\it iii)}.
\end{proof}
\begin{rem}
It might be that $\mathbf m$ is essential  for $\mathcal{U}(\mathfrak n^-)$ and $(\la,\mathbf m)$ is not essential for $V(\la)$, but
$\mathbf f^{(\mathbf m)}v_\la\not=0$ in $V(\la)$.
\end{rem}

\subsection{The monoid property}
\begin{prop}\label{plus}\it
\begin{itemize}
\item[{\it i)}] $es(\mathfrak n^-)$ is a monoid, i.e., $\bf m,m'\in es(\mathfrak n^-)$  implies ${(\bf m+m')}\in es(\mathfrak n^-)$.
\item[{\it ii)}] If $(\lambda, \bf m)$ is essential for $V(\la)$ and $(\mu , \bf m')$ is essential for $V(\mu)$, then
$(\lambda+\mu, \bf m+m')$ is essential for $V(\la+\mu)$.
\end{itemize}
\end{prop}
\begin{proof}
By Lemma~\ref{repessentialversusessential}, it suffices to prove {\it ii)}. The proof of {\it ii)} is the same as in \cite{FFL1},
or one uses the connection with the evaluation monoid proved in Proposition~\ref{semigleichsemi1}.
\end{proof}
\begin{coro}\label{semigroup}\it
The following sets are naturally endowed with the structure of a monoid:
$$
es(\mathfrak n^-)\subseteq \mathbb N^N,\  \Gamma(\la)=\bigcup_{n\in\mathbb N} n\times es(n\la)\subset \mathbb N\times \mathbb N^N\ \hbox{and\ }\
\Gamma=\bigcup_{\la\in\Lambda^+} \{\la\}\times es(\la) \subset \Lambda^+\times \mathbb N^N.
$$
\end{coro}
\begin{defn}\label{Def:Ess}
We call $es(\mathfrak n^-)$ the {\it essential monoid} of $\mathfrak n^-$,
$\Gamma(\la)$ is called the {\it essential monoid associated to $\la$}, and $\Gamma$
is called the {\it global essential monoid}. The real cone ${\cal C}_{(S,>)}=\mathbb R_{\ge 0} \Gamma \subset \Lambda_{\mathbb R}\times \mathbb R^N$
is called the {\it essential cone} associated to $(\mathbb N^N,>,S)$.
\end{defn}

\section{The dual essential basis and the equality of monoids}\label{dualessentialsection}
Starting with a fixed birational sequence $S=(\beta_1,\ldots,\beta_N)$ and a $\Psi$-weighted lex order ``$>$'' on $\mathbb N^N$,
we associate two monoids to $G$:
the global essential monoid $\Gamma$ and the value monoid ${\cal V}(G/\hskip -3.5pt/U)$.
We show that the two monoids can be naturally identified.
\subsection{The dual essential basis}
By construction,  $\mathbb B_\la=\{\mathbf f^{(\mathbf p)}v_\la \mid \bp \in es(\la) \}$ is a basis of $V(\la)$, let $\mathbb B_\la^*\subset V(\la)^*$
be the dual basis. For $\mathbf f^{(\mathbf p)}v_\la\in \mathbb B_\la$ denote by $\xi_{\la,\bp}\in \mathbb B_\la^*$ the corresponding dual element.
\begin{lem}\label{<1a}\it
Let $\bq=(q_i)_{i=1}^N$ be a multi-exponent (not necessarily essential). Then for any $\bp\in es(\la)$ such that $\bq<\bp$ we have
$\xi_{\la,\bp}(\mathbf f^{(\mathbf q)}v_\la)=0$.
\end{lem}
\begin{proof}
The vector $\mathbf f^{(\mathbf q)}v_\la$ can be expressed as a linear combination of vectors $\mathbf f^{(\mathbf q')}v_\la$ with $\bq'\in es(\la)$
and $\bq'\le \bq<\bp$, which proves the lemma.
\end{proof}
Keeping in mind the isomorphisms in \eqref{coordinateGnachU}, the set
$\mathbb B_\Gamma:=\{\xi_{\la,\bf p}\mid (\la,{\bf p})\in \Gamma\}$
is a vector space basis for the algebra $\mathbb C[G/\hskip -3.5pt/U]$, we call it the {\it dual essential basis}.
Consider the structure constants $c_{\la,\bp;\mu,\bq}^{\la+\mu,\br}$, defined for $\bp\in {\rm es}(\lambda)$ and
$\bq\in {\rm es}(\mu)$ by
\[
\xi_{\la,\bp}\xi_{\mu,\bq}=\sum_{\br\in {\rm es}(\lambda+\mu)} c_{\la,\bp;\mu,\bq}^{\la+\mu,\br}\xi_{\la+\mu,\br}.
\]
{
\begin{lem}\label{sc1}\it
$\xi_{\la,\bp}\xi_{\mu,\bq}= \xi_{\la+\mu,\bp+\bq} + \sum_{\br\in {\rm es}(\lambda+\mu),\br >\bp+\bq} c_{\la,\bp;\mu,\bq}^{\la+\mu,\br}\xi_{\la+\mu,\br}$.
\end{lem}}
\begin{proof}
It remains to prove that the structure constant $c_{\la,\bp;\mu,\bq}^{\la+\mu,\br}$ vanishes if $\br<\bp+\bq$ and $c_{\la,\bp;\mu,\bq}^{\la+\mu,\bp+\bq}=1$.
We have
\[
c_{\la,\bp;\mu,\bq}^{\la+\mu,\br}=(\xi_{\la,\bp}\T\xi_{\mu,\bq})(\bff^{(\br)}(v_\la\T v_\mu))=\sum_{\br'+\br''=\br}
d_{\la,\br';\mu,\br''}\xi_\bp(  \bff^{(\br')}v_\la)\xi_\bq(\bff^{(\br'')}v_\mu)),
\]
where $d_{\la,\br';\mu,\br''}$ are some constants (multiplicities of the corresponding terms).
Now if $\br<\bp+\bq$, then either $\br'<\bp$ or $\br''<\bq$ because otherwise
$\bp+\bq\le \br'+\bq\le\br'+\br''=\br$, which is a contradiction. By Lemma \ref{<1a}, the claim follows.
If $\br=\bp+\bq$ and $\br'\ne \bp$, then again either $\br'<\bp$ and $\xi_{\la,\bp}( \bff^{(\br')}v_\la)=0$, or $\br'>\bp$
and then $\br''<\bq=\br-\bp$, so $\xi_{\mu,\bq}( \bff^{(\br'')}v_\mu)=0$. Hence only the terms with $\br'=\bp$, $\br''=\bq$
contribute to $c_{\la,\bp;\mu,\bq}^{\la+\mu,\bp+\bq}$.
Now according to the arguments above, $c_{\la,\bp;\mu,\bq}^{\la+\mu,\bp+\bq}$ is equal to the product over all $i$
of the coefficients of $f_i^{(p_i)} \T f_i^{(q_i)}$ in $(f_i\T 1+ 1\T f_i)^{(p_i+q_i)}$,
but the latter is equal to one.
\end{proof}

\subsection{The global essential monoid and the valuation monoid}
\begin{prop}\label{semigleichsemi1}\it
The global essential monoid $\Gamma$ and the valuation monoid ${\cal V}(G/\hskip -3.5pt/U)$ coincide.
\end{prop}
\begin{proof} We have to determine $\nu(\xi_{\la,\mathbf p})$ for all $\la\in\Lambda^+$ and $\bp\in es(\la)$.
Now by \eqref{coordinatefunction} we have
$$
\begin{array}{rcl}
(\xi_{\la,\mathbf p}\circ\phi)(s_1,\ldots,s_N,t)&=&\xi_{\la,\mathbf p}(\exp(s_1f_1)\cdots \exp(s_Nf_N) t.v_\la) \\
&=&
 \la(t).\xi_{\la,\mathbf p}(\exp(s_1f_1)\cdots \exp(s_Nf_N)v_\la)\\
&=& \la(t).\xi_{\la,\mathbf p}(\sum_{\mathbf m\in\mathbb N^N}\mathbf s^{\mathbf m}\mathbf f^{(\mathbf m)}v_\la).
\end{array}
$$
One can rewrite the last argument as a linear combination of essential vectors and one gets:
$$
(\xi_{\la,\mathbf p}\circ\phi)(s_1,\ldots,s_N,t)=
\la(t).\xi_{\la,\mathbf p}\bigg(\sum_{\mathbf m\in es(\la)}
\big( 
\mathbf s^{\mathbf m}+\sum_{\bq>\mathbf m} c_{\bq,\mathbf m}\mathbf s^{\bq}
\big)
\mathbf f^{(\mathbf m)}v_\la\bigg).
$$
It follows that
\begin{equation}\label{functionpi}
(\xi_{\la,\mathbf p}\circ\phi)(s_1,\ldots,s_N,t)=(
\mathbf x^{\bp} e^\la +\sum_{\bq>\mathbf p} c_{\bq,\mathbf p}\mathbf x^{\bq} e^\la)(s_1,\ldots,s_N,t),
\end{equation}
so $\nu(\xi_{\la,\mathbf p})=(\la,\mathbf p)$ and hence $\Gamma\subseteq {\cal V}(G/\hskip -3.5pt/U)$.
If $f\in \mathbb C[G/\hskip -3.5pt /U]$ is an arbitrary element, then it is of the form $f=\sum_{\mu\in\Lambda^+} \sum_{\mathbf p\in es(\mu)}
b_{\mu,\mathbf p} \xi_{\mu,\mathbf p}$. If $\rho$ is a character of $G$, then $\xi_{\rho,(0,\ldots,0)}=e^{\rho}$
and $e^{\rho}\xi_{\la,\mathbf m}=\xi_{\la+\rho,\mathbf m}$.
So there exists a character $\chi\in\Lambda^\dagger$ of $G$
such that $f=e^{-\chi} (\sum_{\mu'\in\Lambda^\dagger} \sum_{\mathbf p\in es(\mu')} c_{\mu',\mathbf p} \xi_{\mu',\mathbf p})$.
Let $(\la, \mathbf p)$ be the unique minimal tuple (with respect to ``$\triangleright$'', see section~\ref{SValuationsemigroup})
such that $c_{\la,\mathbf p}\not=0$. Then \eqref{functionpi} implies: $f\circ\phi=e^{-\chi}(
c_{\la,\mathbf p}\mathbf x^{\bp} e^\la +\sum_{(\la,\mathbf p)\triangleleft(\mu',\bq)} d_{\mu',\bq}\mathbf x^{\bq} e^{\mu'})$ and hence
$\nu(f)=(\la-\chi,\mathbf p)\in \Gamma$. It follows that $\Gamma= {\cal V}(G/\hskip -3.5pt/U)$.
\end{proof}
\begin{coro}\label{rank}\it
$\langle \Gamma\rangle_{\mathbb Z}=   \Lambda\times\mathbb Z^{N}$ and
$\langle es(\mathfrak n^-)\rangle_{\mathbb Z}= \mathbb Z^{N}$.
\end{coro}
\begin{proof}{
We have a $T$-equivariant birational map $\phi:{\cal Z}_S\rightarrow G/\hskip-3.5pt/U$. It follows
that a rational function $h$ on ${\cal Z}_S$ can be written as a quotient $\frac{f_1}{f_2}\circ\phi$,
where $f_1,f_2\in  \mathbb C[G/\hskip -3.5pt/U]$ are regular functions on $G/\hskip -3.5pt/U$.
It follows that $\nu(h)=\nu(f_1)-\nu(f_2)\in \langle \Gamma\rangle_{\mathbb Z}$.
}
Since $\nu:\mathbb C({\cal Z}_S)\rightarrow \Lambda\times\mathbb Z^{N}$ is surjective, it follows that
$\langle \Gamma\rangle_{\mathbb Z}=   \Lambda\times\mathbb Z^{N}$.
Now for all $\la\in \Lambda^+$, the pair $(\la,(0,\ldots,0))$ is an element of $\Gamma$,
so the second part follows immediately by Lemma~\ref{repessentialversusessential}.
\end{proof}
We close the subsection with remarks on the possible choices for $\Psi$. Denote by ${\bf e}_i\in \mathbb N^N$ the vector
${\bf e}_i=(\underbrace{0,\ldots,0}_{i-1},1,0\ldots,0)$, then $\Psi$ is
determined by $(\Psi_1,\ldots,\Psi_N)\in\mathbb N^N$,
where $\Psi_i:=\Psi({\bf e}_i)$.

\begin{rem}\label{changepsi}
Let $S=(\beta_1,\ldots,\beta_N)$ be a birational sequence and fix a $\Psi$-weighted lex order ``$>$''
on $\mathbb N^N$. If one replaces $\Psi$ by $k\Psi$ for some $k>0$, then the
associated global essential monoid does not change because the filtration
does not change. So one should think of the possible choices for $\Psi$
as rational points on the intersection $\cal W$ of a sphere with the positive quadrant
in $\mathbb R^N$.

Further, let $\tilde \Psi$ be an integral weight function such that  $\tilde \Psi(\mathbf m)$
depends only on the weight $\sum_{i=1}^N m_i\beta_i$. An example for such a function
is given in Example 9. If one replaces in the situation above
$\Psi$ by $\Psi +\tilde\Psi$, then the filtration may change but the associated
global essential monoid does not change, this can be proved using an easy weight argument.

In fact, we conjecture that $\cal W$ admits a finite triangulation such that the global essential monoid
stays constant for any choice of $\Psi$ in the interior of a simplex of the triangulation.
\end{rem}

\begin{rem}
It suffices to consider only (opposite) lexicographic
orders. One can switch from lexicographic orders to right lexicographic orders (and vice versa)
by replacing the birational system $S=(\beta_1,\ldots,\beta_N)$ by  $S'=(\beta_N,\ldots,\beta_1)$,
$\Psi$ by $\Psi'=(\Psi_N,\ldots,\Psi_1)$, then $\mathbf m\in \Gamma$ if and only if
$\mathbf m'=(m_N,\ldots,m_1)\in\Gamma'$.
\end{rem}

\subsection{ASM-sequence and a toric variety}
Let $\mathbb T$ be the torus $(\mathbb C^*)^N$.
The algebra $\mathbb C[\Gamma]$ associated to the monoid $\Gamma$
can be naturally endowed with the structure of a $T\times \mathbb T$-algebra by
\begin{equation}\label{TACTION}
(t,t_1,\dots,t_N)\cdot (\la,\bp):=\la(t)\big(\prod_{i=1}^N t_i^{p_i}\big)(\la,\bp).
\end{equation}
\begin{defn}\label{Def:ASM}
The pair $(S, >)$ is called a {\it sequence with an affine, saturated monoid} (short ASM) if $S=(\beta_1,\ldots,\beta_N)$
is a birational sequence, ``$>$'' is a $\Psi$-weighted lexicographic order and $\Gamma(S, >)$ is an affine, saturated monoid
(in the sense of \cite{CLS}, \S 1).
\end{defn}
\begin{lem}\label{toricvariety}\it
If $\Gamma$ is finitely generated, then
$\hbox{Spec\,} (\mathbb C[\Gamma])$ is naturally
endowed with the structure of a toric variety for $T\times \mathbb T$ of dimension $\dim T+\dim \mathfrak n^-$.
\end{lem}
\begin{proof}
The assumption that $\Gamma\subset \Lambda\times \mathbb Z^{N}$ is finitely generated makes it by construction into an affine semigroup,
hence $\hbox{Spec\,} (\mathbb C[\Gamma])$ is an affine toric variety whose torus has character lattice
$\mathbb Z\cdot \Gamma= \Lambda\times \mathbb Z^{N}$ by Corollary~\ref{rank}, which finishes the proof.
\end{proof}
\subsection{The essential polytope}\label{esspoly}
Let $S$ be a birational sequence and fix a $\Psi$-weighted lex order ``$>$''. Assume that the associated
global essential monoid $\Gamma(S,>)$ is finitely generated, so the essential cone  ${\cal C}_{(S,>)}$ is a rational
polyhedral convex cone.
Let $p:\Lambda_{\mathbb R}\times \mathbb R^N\rightarrow \Lambda_{\mathbb R}$ be the projection onto the first factor.
\begin{defn}
For $\lambda\in \Lambda^+_{\mathbb R}$, the {\it essential polytope} associated to $\lambda$ is defined as
$P(\lambda)=p^{-1}(\lambda)\cap {\cal C}_{(S,>)}$.
\end{defn}
{
\begin{rem}
If $\la\in\Lambda^+$ and $(S,>)$ is a sequence with an ASM, then the integral points of $P(\lambda)$
are exactly the pairs $(\la,\mathbf m)$ such that $\mathbf m$ is an essential multi-exponent for $\la$.
\end{rem} }

\section{Example: sequences with an ASM and the string cone}\label{stringconeproof}
Fix a reduced decomposition $\underline{w}_0=s_{i_1}\cdots s_{i_N}$ of the longest word in the Weyl group of $\mathfrak g$
and let ${C}_{\underline{w}_0}\subset \mathbb R^N$ be the associated string cone defined in \cite{Li1,BZ}
and set $S=({\alpha_{i_1}},\ldots, {\alpha_{i_N}})$. By Example~\ref{monomialszwo}, we know that
$S$ is a birational sequence. Let $\Psi:\mathbb N^N\rightarrow \mathbb N$ be the height weighted function
as in Example~\ref{rootorder}. We fix on $\mathbb N^N$ the associated $\Psi$-weighted opposite lexicographic order.
For an element $\mathbf m\in {C}_{\underline{w}_0}$ denote by $G(\mathbf m)$ the corresponding element
of the global crystal basis of $\mathcal{U}(\mathfrak n^-)$ \cite{Ka}, specialized at $q=1$.
The cone ${\cal C}_{\underline{w}_0}\subset\Lambda_{\mathbb R}\times\mathbb R^N$ (see \cite{Li1,BZ}, compare also \cite{AB})
is defined to be the intersection of $\Lambda_{\mathbb R}\times {C}_{\underline{w}_0}$ with $N$ half-spaces:
\begin{equation}\label{weightequation}
{\cal C}_{\underline{w}_0}=\left\{(\lambda,\mathbf m)\in \Lambda_{\mathbb R}\times {C}_{\underline{w}_0}\mid
m_k \le \langle\lambda,\alpha_{i_k}^\vee\rangle -\sum_{\ell=k+1}^N\langle \alpha_{i_\ell},\alpha_{i_k}^\vee \rangle m_\ell,\quad k=1,\ldots,N.
\right\}
\end{equation}
\begin{thm}\label{stringCone}
$es(\mathfrak n^-)={C}_{\underline{w}_0}\cap \mathbb Z^n$.
\end{thm}
\begin{coro}\it
$(S,>)$ is a sequence with an ASM.
\end{coro}
\vskip 1pt\noindent
{\it Proof of the corollary\/}.
It follows by Theorem~\ref{stringCone}
and \cite{BZ, Li1} that $\Gamma$ is the monoid of integral points of the rational  polyhedral convex cone
${\cal C}_{\underline{w}_0}$, and hence is finitely generated and saturated.
\qed
\vskip 1pt\noindent
{\it Proof of the theorem\/}. Part {\it i)} of Lemma~\ref{stringorder} below shows that the elements $\mathbf f^{(\mathbf m)}$ with
$\mathbf m\not\in {C}_{\underline{w}_0}$ are not essential, so $es(\mathfrak n^-)\subseteq {C}_{\underline{w}_0}$,
and part {\it ii)} shows that the monomials $\mathbf f^{(\mathbf m)}$, $\mathbf m\in {C}_{\underline{w}_0}$, form a basis
of $\mathcal{U}(\mathfrak n^-)$, which implies: $es(\mathfrak n^-)={C}_{\underline{w}_0}\cap \mathbb Z^n$.
\qed

\begin{lem}\label{stringorder}\it
\begin{itemize}
\item[{\it i)}] If $\mathbf m\in \mathbb N^N$ but $\mathbf m\not\in {C}_{\underline{w}_0}$, then $\mathbf f^{(\mathbf m)}$
is a linear combination of a finite number of monomials $\mathbf f^{\mathbf m'}$ such that $\mathbf m'\in {C}_{\underline{w}_0}$
and $\mathbf m'<\mathbf m$.
\item[{\it ii)}] If $\mathbf m\in {C}_{\underline{w}_0}$, then $\mathbf f^{(\mathbf m)}=\sum_{ \mathbf m'\le \mathbf m} a_{\mathbf m'}G(\mathbf m')$
such that $a_{\mathbf m}=1$.
\end{itemize}
\end{lem}
Before we come to the proof, note that as an immediate corollary one gets:
\begin{coro}\it
The $(\mathbb N^N,>,S)$-filtration of $\mathcal{U}(\mathfrak n^-)$ is compatible with Kashiwara's global crystal basis.
In particular, given a dominant weight $\la$, $\mathbf m\in   \mathbb N^N$ is an essential
multi-exponent for $V(\la)$ if and only if $G(\mathbf m).v_\la\not=0$.
\end{coro}
\vskip 1pt\noindent
{\it Proof of the lemma\/}.
The string cone is defined more generally for a reduced decomposition $w=s_{i_1}\cdots s_{i_q}$ of an arbitrary
element of $W$, one has ${C}_{\underline{w}}\subset \mathbb N^q$. Moreover, if one completes the reduced decomposition
of $w$ to a reduced decomposition of $w_0=s_{i_1}\cdots s_{i_q}\cdot s_{i_{q+1}}\cdots s_{i_N}$, then
${C}_{\underline{w}}\subseteq {C}_{\underline{w}_0}$ is the subset of $N$-tuples such that $m_{q+1}=\ldots,m_N=0$
(see \cite{Li1}).

The monomials $\mathbf f^{(\mathbf m)}$ are eigenvectors for the action of $T$. So given a linear dependence relation
between some monomials, we may assume that all monomials occurring in the relation
have the same $T$-weight. So the only relevant order is the opposite lexicographic order because in this case we have $\Psi(\mathbf m)=\Psi(\mathbf m')$
for all terms occurring in such a linear dependence relation.

To prove a linear dependence relation or  a linear independence relation between monomials in $\mathcal{U}(\mathfrak n^-)$, it is sufficient
to do this for the terms applied to a highest weight vector $v_\la$ for $\la$ sufficiently large.
We often use  the fact that the global crystal basis for the corresponding quantum group and all necessary structure constants
are defined over $\mathbb C[q,q^{-1}]$, so one can specialize the basis at $q=1$ to conclude results for $\mathcal{U}(\mathfrak n^-)$.
For details see \cite{Ka},\cite{Ha} or \cite{Li1}. Keeping this in mind, part {\it ii)} has been already proved
in \cite{Li1}, Proposition 10.3, for $\mathbf f^{(\mathbf m)}$, $\mathbf m\in {C}_{\underline{w}}\subset \mathbb N^q$.
It remains to prove part {\it i)}.

We will prove the Weyl group element version of {\it i)} by induction on the length of $w$, the result for $w_0$ implies
then part {\it i)} above. If $\ell(w)=1$ and $w=s_{i}$, nothing has to be proven: all $\mathbf m=(m_1)$
are elements of ${C}_{\underline{w}}$. Suppose now $w=s_{i_1}\cdots s_{i_{j+1}}$ for some $j\ge 1$,
and we assume that the lemma holds for all elements $w'\in W$ of length smaller than or equal to $j$ and for all possible
reduced decompositions of these elements.
Suppose that $\mathbf m\in \mathbb N^{j+1}$ but  $\mathbf m\not\in {C}_{\underline{w}}$. Then $w'=s_{i_2}\cdots s_{i_{j+1}}$
is a reduced decomposition of an element of length $j$. So if $\mathbf m'=(m_2,\ldots,m_{j+1})\not \in {C}_{\underline{w'}}$,
then by induction we know that $\mathbf f^{(\mathbf m')}$ is a linear combination of a finite number of monomials
$\mathbf f^{\mathbf k'}$ such that $\mathbf k'\in {C}_{\underline{w'}}$
and $\mathbf k'<\mathbf m'$ and hence
$$
\mathbf f^{(\mathbf m)}=f_{i_1}^{(m_1)}\mathbf f^{(\mathbf m')}=
\sum_{\substack{\mathbf k'\in\mathbb N^{j}\\ \mathbf{k'<m'}}} a_{\mathbf{k',m'}}f_{i_1}^{(m_1)} \mathbf f^{(\mathbf k')}
=
\sum_{\substack{\mathbf k\in\mathbb N^{j+1}\\ \mathbf{k<m}}} a_{\mathbf{k,m}}f^{(\mathbf k)},
$$
where the $\mathbf k$'s are obtained from the $\mathbf k'$'s by adding $m_1$ at the beginning.
So assume now $\mathbf m\in \mathbb N^{j+1}$,  $\mathbf m\not\in {C}_{\underline{w}}$ but
$\mathbf m'=(m_2,\ldots,m_{j+1}) \in {C}_{\underline{w'}}$. Now we can apply part {\it ii)} in the general form mentioned above:
$$
\mathbf f^{(\mathbf m')}=\sum_{\substack{\mathbf k' \in {\cal C}_{\underline{w'}}\\ \mathbf k'\le \mathbf m'}} a_{\mathbf{k',m'}}G(\mathbf k'),
$$
where $a_{\mathbf{m',m'}}=1$. Since the transition matrix in the part {\it ii)} (also in its more general form, see \cite{Li1}) is upper triangular, it follows
that $G(\mathbf k')$ is a linear combination of $\mathbf f^{(\mathbf h')}$ such that $\mathbf h'\le  \mathbf k'\le \mathbf m'$.
For $\mathbf k'< \mathbf m'$ it follows that
$$
f_{i_1}^{(m_1)}G(\mathbf k')=\sum_{\substack{\mathbf h' \in {\cal C}_{\underline{w'}}\\ \mathbf h'\le \mathbf k'}}
a_{\mathbf{k',m'}} f_{i_1}^{(m_1)}\mathbf f^{(\mathbf h')}
=\sum_{\substack{\mathbf h\in\mathbb N^{j+1}\\ \mathbf{h<m}}} a_{\mathbf{h,m}}f^{(\mathbf h)},
$$
where the $\mathbf h$ are obtained from the $\mathbf h'$ by adding a $m_1$ as a first component.
In particular, we get $f_{i_1}^{(m_1)}G(\mathbf k')$ as a linear combination of monomials strictly smaller than $\mathbf f^{(\mathbf m)}$.

It remains to consider the term $f_{i_1}^{(m_1)}G(\mathbf m')$. Suppose first $(0,m_2,\ldots,m_q)\in {C}_{\underline{w}}$.
Since $\mathbf m\not\in {C}_{\underline{w}}$, Proposition 6.4.3 in \cite{Ka}
implies that $f_{i_1}^{(m_1)}G(\mathbf m')$ is a linear combination of global basis elements $G(\mathbf b)$ such that $b_1>m_1$
and hence $\mathbf b < \mathbf m$. Again applying part {\it ii)}, we get $f_{i_1}^{(m_1)}G(\mathbf m')$ as a linear combination
of monomials $\mathbf f^{\mathbf h}$ such that $\mathbf{h<m}$.

Suppose now $(0,m_2,\ldots,m_q)\not \in {C}_{\underline{w}}$, but recall that we assume $(m_2,\ldots,m_q) \in { C}_{\underline{w'}}$.
Let $\mathbf a=(a_1,\ldots,a_q)$ be the corresponding
string parametrization of the global basis element for the fixed decomposition of $w$. Recall that the global basis is independent of the parametrization,
so $G(\mathbf m')=G(\mathbf a)$. Proposition 6.4.3 in \cite{Ka}
implies that $f_{i_1}^{(m_1)}G(\mathbf m')=f_{i_1}^{(m_1)}G(\mathbf a)=$  a linear combination of global basis elements
$G(\mathbf b)$ such that $b_1\ge a_1+m_1$. Since $a_1>0$ by assumption, it follows that
$\mathbf b < \mathbf m$. Again applying part {\it ii)}, we get $f_{i_1}^{(m_1)}G(\mathbf m')$ as a linear combination
of monomials $\mathbf f^{\mathbf h}$ such that $\mathbf{h<m}$.

Summarizing: if $\mathbf m\not\in  {C}_{\underline{w}}$, then $\mathbf f^{(\mathbf m)}$ is a linear combination
of monomials $\mathbf f^{(\mathbf m')}$ such that $\mathbf m' < \mathbf m$, and all are of the same $T$-weight.
Since there is only a finite number of ordered monomials of a given weight, it follows that, after repeating the procedure if necessary,
one can rewrite $\mathbf f^{(\mathbf m)}$ as a linear combination
of monomials $\mathbf f^{(\mathbf m')}$ such that $\mathbf m' < \mathbf m$ and $\mathbf m'\in{C}_{\underline{w}}$.
\qed
\section{Example: sequences with an ASM and Lusztig's parametrization}\label{Lusztig}
Fix a reduced decomposition $\underline{w}_0=s_{i_1}\cdots s_{i_N}$ of the longest word $w_0$ in the Weyl group $W$ of $\mathfrak g=\hbox{Lie G\,}$
and let $S=(\beta_1,\ldots,\beta_N)$ be an enumeration of the positive roots
associated to the decomposition, i.e., $\beta_k=s_{i_1}\cdots s_{i_{k-1}}(\alpha_k)$ for $k=1,\ldots,N$. 
By Example~\ref{monomialseins} we know that
$S$ is a birational sequence. Let $\Psi:\mathbb N^N\rightarrow \mathbb N$ be the height weighted function
as in Example~\ref{rootorder}. We fix on $\mathbb N^N$ the associated  $\Psi$-weighted opposite right
 lexicographic order ``$>$''.
For an element $\mathbf m\in \mathbb N^N$ denote by $B(\mathbf m)$ the corresponding element
of Lusztig's canonical basis of $\mathcal{U}(\mathfrak n^-)$ (using Lusztig's parametrization with respect to the fixed decomposition $\underline{w}_0$),
specialized at $q=1$.

The decomposition determines also an enumeration of the positive roots (see the choice of $S$ above) and hence a PBW-basis.
The connection between this PBW-basis and the canonical basis was described by Lusztig in the $\tt A$, $\tt D$ and $\tt E$-case
\cite{Lu1,Lu2}, and by Caldero  \cite{C2} in the semisimple case. In terms of the $(\mathbb N^N,>,S)$-filtration
this can be reformulated as follows:
\begin{thm}\it
The $(\mathbb N^N,>,S)$-filtration of $\mathcal{U}(\mathfrak n^-)$ is compatible with Lusztig's canonical basis.
In particular, given a dominant weight $\la$, $\mathbf m\in   \mathbb N^N$ is an essential
multi-exponent for $V(\la)$ if and only if $B(\mathbf m).v_\la\not=0$.
\end{thm}
\begin{coro}\it
$(S,>)$ is a sequence with an ASM.
\end{coro}
\vskip 1pt\noindent
{\it Proof of the theorem}.
By \cite{C2}, Corollary 2.1, one knows:
\begin{equation}\label{Eq2}
B(\mathbf m) = f^{(\mathbf m)} +\sum_{\mathbf m<_{rlex} \mathbf n} d_{\mathbf n}^{\mathbf m}(q)   f^{(\mathbf n)},
\end{equation}
where $d_{\mathbf n}^{\mathbf m}(q)\in q\mathbb Z[q]$.
So we get:
$$
f^{(\mathbf m)} = B (\mathbf m) + \sum_{\mathbf m<_{rlex} \mathbf n} c_{\mathbf n}^{\mathbf m}(q)  f^{(\mathbf n)}.
$$
One can replace the $f^{(\mathbf n)}$ by similar expressions. Since all monomials have the same $T$-weight, there is only a finite number
of them which are strictly larger than $\mathbf m$ with respect to ``$<_{rlex} $''. So by repeating the procedure if necessary, we get
$$
f^{(\mathbf m)} = B (\mathbf m) + \sum_{\mathbf m<_{rlex} \mathbf n} b_{\mathbf n}^{\mathbf m}(q)  B (\mathbf n).
$$
Let $\Psi$ be the height function as in Example~\ref{rootorder} and use as order ``$>$" the $\Psi$-weighted opposite right lexicographic order.
Specialized at $q=1$, one can reformulate the above as follows:
$$
f^{(\mathbf m)} = B (\mathbf m) + \sum_{\bf m> n} b_{\mathbf n}^{\mathbf m}(1)  B(\mathbf n).
$$
So the filtration of $\mathcal{U}(\mathfrak n^-)$ by the PBW basis with respect to ``$>$'' is compatible with Lusztig's canonical basis.
In particular, given a dominant weight $\lambda$ and a highest weight vector $v_\la$, then the properties of the good basis imply:
$$
f^{(\mathbf m)} v_\la \hbox{\ is essential for\ }V(\lambda)\Longleftrightarrow B(\mathbf m)v_\lambda\not=0.
$$
\qed
\vskip 1pt\noindent
{\it Proof of the corollary}.
It remains to show that $\Gamma$ is finitely generated.
If $G$ is of type $\tt A$, $\tt D$ or $\tt E$, then Lusztig describes
\cite{Lu1,Lu2} the elements $B(\mathbf m)$ of the canonical basis such that $B(\mathbf m).v_\la\not=0$.
From this description it follows immediately that $\Gamma$ is the set of integral points in a finite union
of closed rational polyhedral cones, and hence is finitely generated and saturated.

For the general case we refer to \cite{BZ}, where
piecewise linear maps $\rho(\underline{w}_0,\underline{w}'_0):\mathbb R_{\ge 0}^N\rightarrow \mathbb R_{\ge 0}^N$
are described
which transform Lusztig's parametrization with respect to the reduced decomposition $\underline{w}_0$
into the string parameter with respect to a reduced decomposition $\underline{w}'_0$. Recall that the string
cone is a closed rational polyhedral cone. Again it follows
that $\Gamma$ is the set of integral points in a finite union
of closed rational polyhedral cones, and hence is finitely generated and saturated.
\qed

As in the string cone case, it seems to be difficult to give an explicit description for the cones and polytopes, except
in some special cases. In the case of the string cone, examples can be found in \cite{Li1} and in \cite{AB}.
We explain now a useful tool to calculate the essential polytope in terms of the
Lusztig parametrization for $G=SL_{n+1}(\mathbb C)$. We fix the following reduced decomposition:
\begin{equation}\label{AnSpecial}
\underline{w}_0=(s_ns_{n-1}\ldots s_1)(s_{n}s_{n-1}\ldots s_2)\ldots (s_ns_{n-1})s_n,
\end{equation}
$\Psi$ is the height function and ``$>$'' is the associated  $\Psi$-weighted opposite right lexicographic order.

Let $\operatorname{deg}$ be the homogeneous degree function on $\mathbb{N}^N$.
Let $Q$ be the equi-oriented quiver of type $A_n$ where the vertex $1$ is a source. We fix a reduced decomposition
which is compatible with the orientation of the quiver. The indecomposable representation of $Q$ corresponding to the positive root $\alpha_{i,j}$ will be denoted by $M_{i,j}$.

\begin{prop}\it Let $G=SL_{n+1}(\mathbb C)$ and let $S=(\beta_1,\ldots,\beta_N)$ be according to the decomposition in \eqref{AnSpecial}.
If $\bold{m}$ and $\bold{n}$ are two multi-exponents satisfying $wt(\bold{m})=wt(\bold{n})$ such that $f^{(\bold{m})}v_\lambda$ is essential for $V(\lambda)$ but $f^{(\bold{n})}v_\lambda$ is not, then $\operatorname{deg}(\bold{n})\leq\operatorname{deg}(\bold{m})$.
\end{prop}

The remaining part of this section is devoted to the proof of this proposition.

\begin{lem}\it
For any $\bold{n}$ in equation \eqref{Eq2} such that $d_{\bold{n}}^{\bold{m}}(q)\neq 0$, $\operatorname{deg}(\bold{n})\geq\operatorname{deg}(\bold{m})$.
\end{lem}

\begin{proof}
The proof of this lemma uses the Hall algebra of $Q$.
We fix the following degeneration order on the isoclasses of representations of $Q$ having the same fixed dimension vector:
$[M]\leq [N]$ if and only if for any indecomposable representation $V$ one has
$\dim\operatorname{Hom}(V,M)\leq \dim\operatorname{Hom}(V,N).$

We fix the following decompositions of $M$ and $N$ into indecomposables:
$$M=\bigoplus_{1\leq i\leq j\leq n}M_{i,j}^{\oplus m_{i,j}},\ \ N=\bigoplus_{1\leq i\leq j\leq n}M_{i,j}^{\oplus n_{i,j}}.$$
Let $E_{[M]}$ be the PBW root vector associated to the isoclass $[M]$. For a representation $M$ as above, we define the PBW degree of $E_{[M]}$ to be $$\operatorname{deg}(E_{[M]})=\sum_{1\leq i\leq j\leq n}m_{i,j}.$$
\par
We claim that $[M]\leq [N]$ implies $\operatorname{deg}(E_{[M]})\leq \operatorname{deg}(E_{[N]})$. Indeed, taking $r=s=k$ in the following formula
$$\dim\operatorname{Hom}(M_{r,s},M_{i,j})=\left\{\begin{matrix} 1,& 1\leq i\leq r\leq j\leq s\leq n\\
0, & \text{otherwise}\end{matrix}\right.$$
gives $\dim\operatorname{Hom}(M_{k,k},M)=\sum_{1\leq i\leq k}m_{i,k}$, which proves the claim.

Let
$\{\bold{b}_{[M]}\ |\ [M]\ \text{is an isoclass of representation of }Q\}$
be Lusztig's canonical basis. By Section 7.11 of [Lu1],
\begin{equation}\label{Eq}
\bold{b}_{[M]}=E_{[M]}+\sum_{[M']>[M]}\zeta_{[M']}^{[M]}E_{[M']},
\end{equation}
where $<$ is the degeneration order. As a corollary of the above claim, in (\ref{Eq}), all $E_{[M']}$ appearing in the sum satisfy: $\operatorname{deg}(E_{[M']})\geq \operatorname{deg}(E_{[M]})$.
\par
In the isomorphism between the Hall algebra of $Q$ and the negative part of the corresponding quantum group, up to some scalar, the PBW root vector $E_{[M]}$ corresponds to $f^{(\bold{m})}$ where
$$\bold{m}=(m_{n,n},m_{n-1,n},\ldots,m_{1,n},m_{n-1,n-1},\ldots,m_{1,n-1},\ldots,m_{2,2},m_{1,2},m_{1,1}).$$
This proves the lemma.
\end{proof}

\begin{proof}[Proof of the proposition]
By the above lemma, in formula (\ref{Eq2}), for any multi-exponent $\bold{n}$ in the sum, $\operatorname{deg}(\bold{n})\geq\operatorname{deg}(\bold{m})$. So in the proof of Theorem 2 we can take this homogeneous degree into account, proving that essential monomials have larger homogeneous degree.
\end{proof}

\section{Example: a sequence with an ASM in the homogenous PBW-type case}\label{homordercase}
In this section let $G=SL_{n+1}$ or $Sp_{2n}$. Let $\alpha_i=\epsilon_i-\epsilon_{i+1}$,
$i=1,\ldots,n$ be the simple roots for $SL_{n+1}$, in the symplectic case the simple roots are
$\alpha_i=\epsilon_i-\epsilon_{i+1}$, $i=1,\ldots,n-1$, and $\alpha_n=2\epsilon_n$.
In the $SL_{n+1}$-case, the positive roots are of the form $\al_{i,j}=\al_i+\dots +\al_j$ for $1\le i\le j\le n$.
In the symplectic case the positive roots are of the form
\begin{gather*}
\al_{i,j}=\al_i+\al_{i+1}+\dots +\al_j,\ 1\le i\le j\le n,\\
\alpha_{i, \ol{j}} = \alpha_i + \alpha_{i+1} + \ldots +
\alpha_n + \alpha_{n-1} + \ldots + \alpha_j, \ 1\le i\le j\le n
\end{gather*}
(note that $\al_{i,n}=\al_{i,\ol n}$).

 In the $SL_{n+1}$-case, a sequence ${\bf b}=(\delta_1,\dots,\delta_r)$ of
positive roots is called a {\it Dyck path} if the first and the last roots are
simple roots ($\delta_1=\al_{i,i}$, $\delta_r=\al_{j,j}$, $i\le j$), and if $\delta_m=\al_{p,q}$, then $\delta_{m+1}=\al_{p+1,q}$ or
$\delta_{m+1}=\al_{p,q+1}$.

A {\it symplectic Dyck path} is a sequence ${\bf b}=(\delta_1,\dots,\delta_r)$ of positive roots (for the symplectic group) such that:
the first root is a simple root, $\beta_1=\al_{i,i}$; the last root is either a simple root $\beta_r= \al_j$ or
$\beta_r = \al_{j,\ol{j}}$ ($i \le j \leq n$); if $\beta_m=\al_{r,q}$ with $r, q \in A$ then $\beta_{m+1}$ is
either $\al_{r,q+1}$ or $\al_{r+1,q}$, where $x+1$ denotes the smallest element in $A$ which is bigger than $x$.
Here we use the usual order on the alphabet $A = \{1, \ldots, n, \ol{n-1}, \ldots, \ol{1}\}$: $ 1 <2 < \ldots < n-1 < n < \ol{n-1} < \ldots < \ol{1}.$

Let ``$\succ$'' be the standard partial order on the set of positive roots for $G=SL_{n+1}$ respectively $Sp_{2n}$.
We fix an ordering on the positive roots $\beta_1,\dots,\beta_N$ on the positive roots and we assume that
\[
\beta_i\succ \beta_j \text{ implies } i<j.
\]
An ordering with this property (the larger roots come first) will be called a {\it good ordering}.
Once we fix such a good ordering, this induces an ordering on the basis vectors $f_\beta$.
As monomial order ``$>$'' on the PBW basis we fix the induced homogeneous right lexicographic order.

The following description of the global essential monoid $\Gamma$ has been given in \cite{FFL1,FFL2,FFL3}.
Note that the description is independent of the choice of the ordering of the roots, as long as the ordering is a ``{\it good\/}'' ordering.

\begin{thm}\label{theoremthree}\it
The global essential monoid $\Gamma$ is the monoid of integral points in the closed rational polyhedral cone ${\mathcal C}$
defined by the inequalities:

If $G=SL_{n+1}(\mathbb C)$, then
$$\small
{\mathcal C}=\left\{ (\mathbf x,\mathbf m) \in \mathbb R_{\ge 0}^n\times  \mathbb R_{\ge 0}^N \mid
\begin{array}{c}
\hbox{$\forall i=1,\ldots,n$, $\forall$ Dyck paths ${\bf b}=(\delta_1,\dots,\delta_r)$ starting in $\alpha_i$,}\\
\hbox{ending in $\alpha_j: \sum_{\ell=1}^r m_{\delta_\ell}\le x_i+\ldots+x_j$}.\\
\end{array}
\right\}
$$
If $G=Sp_{2n}(\mathbb C)$, then
$$\small
{\mathcal C}=\left\{ (\mathbf x,\mathbf m) \in \mathbb R_{\ge 0}^n\times  \mathbb R_{\ge 0}^N \mid
\begin{array}{c}
\hbox{$\forall i=1,\ldots,n$, $\forall$ symplectic Dyck paths ${\bf b}=(\delta_1,\dots,\delta_r)$ }\\
\hbox{starting in $\alpha_i$, ending in $\alpha_j: \sum_{\ell=1}^r m_{\delta_\ell}\le x_i+\ldots+x_j$}\\
\hbox{$\forall i=1,\ldots,n$, $\forall$ symplectic Dyck paths ${\bf b}=(\delta_1,\dots,\delta_r)$ }\\
\hbox{starting in $\alpha_i$, ending in $\alpha_{j,\ol{j}}: \sum_{\ell=1}^r m_{\delta_\ell}\le x_i+\ldots+x_n$}.
\end{array}
\right\}
$$
In particular, $\Gamma$ is finitely generated and saturated.
\end{thm}

\begin{rem}
A global essential monoid for $G$ of type ${\tt B}_3$, ${\tt D_4}$ and type ${\tt G}_2$ in this PBW-type case can be found in \cite{BK},
\cite{Gor1} and \cite{Gor2} respectively. In other cases than $SL_{n+1}(\mathbb C)$, $Sp_{2n}(\mathbb C)$, ${\tt B}_3$, ${\tt D}_4$ and ${\tt G}_2$,
there is no global essential monoid known in the homogeneous case so far. Partial results, for the span of certain (fundamental) weights in
$\Gamma$, are due to \cite{BD} for minuscule weights and \cite{BK} in the type ${\tt B}_n$. For $G$ of type ${\tt A}_n$, the notion of \textit{triangular} Weyl group elements has been introduced in \cite{F}. For the homogeneous PBW-type case, global essential monoids for the Schubert varieties associated to these Weyl group elements are given in \cite{F}.
\end{rem}

\section{More examples}
In several cases, when the weight function $\Psi$ is properly chosen, the set of essential multi-exponents remains the same for arbitrary refinements.
\begin{exam}\label{quantumex}
Let $G=SL_{n+1}$, we label the positive roots as in the section before: for $1\leq i\leq j\leq n$,
$\alpha_{i,j} = \alpha_i + \ldots + \alpha_j$. Set $N=\frac{1}{2}n(n+1)$, and corresponding to the root
$\alpha_{i,j}$ let $e_{i,j}\in\mathbb N^N$ be the canonical basis vector
$$
e_{i,j}=(\underbrace{0,\ldots,0}_n,\ldots,\underbrace{0,\ldots,0}_{n-i+2},\underbrace{0,\ldots,0}_{j-i},1,0,\ldots,0).
$$
As a matter of fact, any other bijection between the set of positive roots and the canonical basis of $\mathbb N^N$
will give the same result, we just fix one for the convenience of the reader.
We consider the following integral weight function
\[
\Psi(e_{i,j}) = (j - i + 1)(n -j + 1).
\]
Let ``$>$'' be one of the total orders refining the induced partial order ``$>_{\Psi}$'' (Definition~\ref{weightorder}). Again,
as for the choice of the bijection above, the result is the same for any choice. It has been shown in \cite{FFR15}:
\begin{thm}\it
The global essential monoid $\Gamma$ is the monoid of integral points in the closed rational polyhedral cone
${\mathcal C}$ defined in Theorem~\ref{theoremthree} in the $SL_{n+1}$-case.
\end{thm}
\begin{rem}
Actually, the weight function $\Psi$ defined above is just a special example of a larger class of
integral weight functions defined in \cite{FFR15}. This class of integral weight functions is called {\it strongly admissible} in \cite{FFR15}.
In the case of the equioriented quiver of type ${\tt A}$, all strongly admissible functions give rise to the same global
essential monoid described in  Theorem~\ref{theoremthree} in the $SL_{n+1}$-case.

\end{rem}
\begin{rem}
By considering the integral weight function $\Psi$ in the above example, it is shown in \cite{FFR15} that the negative part of the quantum
group of type $\tt A_n$ becomes an $\mathbb{N}$-filtered algebra, and this filtration is compatible with the dual canonical basis (see section 3.4 therein).
\end{rem}
\end{exam}
\begin{exam}
Let $\mathfrak{g}=\mathfrak{sp}_4$ be of type ${\tt C}_2$ and $w:\Phi^+\rightarrow\mathbb{N}_{>0}$ be the function defined by:
$$w(\alpha_{1,1})=1,\ \ w(\alpha_{1,2})=1,\ \ w(\alpha_{1,\overline{1}})=1, \ \ w(\alpha_2)=2.$$
For any sequence $S=(\beta_1,\beta_2,\beta_3,\beta_4)$ where $\Phi^+=\{\beta_1,\beta_2,\beta_3,\beta_4\}$,
we define an integral weight function $\Psi_w$ by: for $\bold{m}=(m_1,m_2,m_3,m_4)$,
$$\Psi_w(\bold{m})=\sum_{i=1}^4 m_i w(\beta_i).$$
It is shown in \cite{BFF15} that for an arbitrary sequence of positive roots $S=(\beta_1,\ldots,\beta_4)$ where $\Phi^+=\{\beta_1,\ldots,\beta_4\}$,
the global essential monoid $\Gamma$ associated to $(S, >)$ does not depend on the choice of $S$. Up to a permutation of
coordinates, it admits the following description: the essential multi-exponents are lattice points in the polyhedral cone in $\Lambda_\mathbb{R}\times \mathbb{R}^4\subset \mathbb{R}^6$ with coordinates $(\ell_1,\ell_2,x_1,x_2,x_3,x_4)$:
$$\ell_1,\ \ell_2,\ x_1,\ x_2,\ x_3,\ x_4\geq 0,\ \ x_1\leq \ell_1,\ \ x_4\leq \ell_2,$$
$$2x_1+x_2+2x_3+2x_4\leq 2(\ell_1+\ell_2),$$
$$x_1+x_2+x_3+2x_4\leq \ell_1+2\ell_2.$$

Let $\pi_1:\Lambda_\mathbb{R}\times \mathbb{R}^4\rightarrow\Lambda_\mathbb{R}$ and $\pi_2:\Lambda_\mathbb{R}\times \mathbb{R}^4\rightarrow \mathbb{R}^4$ be the corresponding linear projections. For $\lambda_1,\lambda_2\in\mathbb{N}$, we define
$$\operatorname{SP}_4(\lambda_1,\lambda_2):=\pi_2(\pi_1^{-1}((\lambda_1,\lambda_2))\cap\Gamma(\mathfrak{n}^-))\subset\mathbb{R}^4.$$
Up to the permutation of the coordinates $x_2$ and $x_3$, the polytope $\operatorname{SP}_4(\lambda_1,\lambda_2)$ is the same as that in \cite{Kir1}, Proposition 4.1, which is unimodularly equivalent to the Newton-Okounkov body of some valuation arising from inclusions of (translated) Schubert varieties.
\par
This polytope $\operatorname{SP}_4(\lambda_1,\lambda_2)$ is unimodularly equivalent to neither the Gelfand-Tsetlin polytope of $\mathfrak{sp}_4$ nor
the polytope of $\mathfrak{sp}_4$ described in Theorem~\ref{theoremthree}.
\end{exam}

\begin{exam}
Let $G=Sp_{2n}$, we label the positive roots as in the section before:
for $1 \leq i \leq j \leq n$, $\alpha_{i,j} = \alpha_i + \ldots + \alpha_j$ and $\alpha_{i,\overline{j}} = \alpha_i + \ldots + \alpha_n + \ldots + \alpha_j$. Set $N=n^2$, and let us fix a bijection between the set of positive roots and the canonical basis of $\mathbb N^N$.
We consider the following integral weight function
\[
\Psi(e_{i,j}) = (2n-j)(j-i+1)\ \text{  and  }\ \Psi(e_{i,\overline{j}}) = j(2n - i - j +1).
\]
Let ``$>$'' be one of the total orders refining the induced partial order ``$>_{\Psi}$'' (Definition~\ref{weightorder}). Again,
as for the choice of the bijection above, the result is the same for any choice. It has been shown in \cite{BFF15}:
\begin{thm}\it
The global essential monoid $\Gamma$ is the monoid of integral points in the closed rational polyhedral cone
${\mathcal C}$ defined in Theorem~{\rm \ref{theoremthree}} in the $Sp_{2n}$-case.
\end{thm}
\end{exam}

\section{Degenerations of spherical affine $G$-varieties}\label{AffG}
Throughout this section, we fix a sequence $(S,>)$ with an ASM,
so $\Gamma$ is assumed to be finitely generated and saturated.
Under this assumption, the methods developed by Caldero \cite{C1} and by Alexeev and Brion in \cite{AB}
apply also to this situation. One could in fact just refer to the articles, but to make the paper more self contained,
we recall most of the arguments, but without proofs.

\subsection{Coordinate rings and filtrations}
Let $Y$ be an irreducible affine $G$-variety with coordinate ring $R=\mathbb C[Y]$. Let $R=\bigoplus_{\la\in\Lambda^+} R_{\la}$
be the isotypic decomposition. Denote by $R^U$ the subring of $U$-invariant functions and let $R^G$ be the ring of $G$-invariant
functions. Recall that both rings, $R^U$ and $R^G$, are finitely generated. Let $R^U=\bigoplus_{\la\in\Lambda^+} R^U_\la$ be the decomposition into
$T$-weight spaces, recall that $R^U_\la$ is always a finitely generated $R^G=R^U_0$-module. For a dominant weight $\la$ denote by $\la^*$
the highest weight of the dual representation $V(\la)^*$ of $V(\la)$. One has a canonical
isomorphism of $R^G$-$G$-modules:
 \begin{equation}\label{Coordringdecomp}
R\simeq \bigoplus_{\la\in\Lambda^+} \hbox{Hom}_G(V(\la)^*, R)\otimes V(\la)^*\simeq \bigoplus_{\la\in\Lambda^+} R^U_{\la^*}\otimes V(\la)^*,
\end{equation}
where an element $\phi\otimes \xi\in \hbox{Hom}_G(V(\la)^*, R)\otimes V(\la)^*$ is mapped onto $\phi(\xi)\in R$ respectively onto
$\phi(\xi_\la)\otimes \xi\in R^U_{\la^*}\otimes V(\la)^* $. Here $\xi_\la$ is a fixed highest weight vector in $V(\la)^*$.
Let  $\succeq_{wt}$ be the usual partial order on the set of weights, we filter $R$ by $G$-stable subspaces: for $\la\in\Lambda^+$ set
\begin{equation}\label{weakfiltration}
R_{\preceq_{wt} \la}=\bigoplus_{\substack{\mu\in\Lambda^+\\ \mu\preceq_{wt} \la}} R^U_{\mu^*}\otimes V(\mu)^*.
\end{equation}
The usual decomposition rules for tensor products imply
$R_{\preceq_{wt} \la}\cdot R_{\preceq_{wt} \mu}\subseteq R_{\preceq_{{wt}} \la+\mu}$.
We want to define a finer filtration, it will depend on the type of the order we have chosen.

\subsection{Refining the filtration in \eqref{weakfiltration}} \label{algfilt}
Let ``$>$" be the fixed $\Psi$-weighted lex order on $\mathbb N^N$. We
define a new partial order ``$>_{alg}$" on $\Lambda^+\times\mathbb N^N$ as follows:
\begin{defn}\label{algorder}
$(\la,\bp)>_{alg}(\mu,\bq)\hbox{\ if and only if either\ }\la\succ_{wt} \mu\hbox{\ or\ }\la=\mu \hbox{\ and\ }\bp<\bq$.
\end{defn}
Note that we turn around ``$>$" when going to the coordinate ring.
We refine the filtration in \eqref{weakfiltration} as follows: for $\la\in\Lambda^+$ and $\bp\in es(\la)$ set
\begin{equation}\label{deg1}
R_{\le_{alg}(\la,\bp)}
=\bigoplus_{\substack{(\mu,\bq)\in\Gamma\\ (\mu,\bq)\le_{alg}(\la,\bp)}} R^U_{\mu^*}\otimes \xi_{\mu,\bq}
\hbox{\ and\ }
R_{<_{alg}(\la,\bp)}=\bigoplus_{\substack{(\mu,\bq)\in\Gamma\\(\mu,\bq)<_{alg}(\la,\bp)}} R^U_{\mu^*}\otimes \xi_{\mu,\bq}.
\end{equation}
\begin{lem}\label{filtR}
$$
R_{\le_{alg}(\la,\bp)}R_{\le_{alg}(\mu,\bq)}\subseteq R_{\le_{alg}(\la+\mu,\bp+\bq)},\quad R_{\le_{alg}(\la,\bp)}R_{<_{alg}(\mu,\bq)}\subseteq R_{<_{alg}(\la+\mu,\bp+\bq)}.
$$
\end{lem}
\begin{proof}
Let $f_\eta\in  R^U_{\eta^*}$, $f_{\eta'}\in  R^U_{{\eta'}^*}$ and $\xi_{\eta,\br}\in V(\eta)^*$, $\xi_{\eta',\br'}\in V(\eta')^*$ for $\eta,\eta'\in\Lambda^+$
and $\br\in es(\eta)$, $\br'\in es(\eta')$. Then the tensor product rules and Lemma~\ref{sc1} imply:
$$
(f_\eta\otimes \xi_{\eta,\br})(f_{\eta'}\otimes \xi_{\eta',\br'})\in (f_\eta f_{\eta'})\otimes \xi_{\eta+\eta',\br+\br'} +
\sum_{\substack{\mathbf t\in es(\eta+\eta')\\ \mathbf t>\br+\br'}}
\mathbb C(f_\eta f_{\eta'})\otimes \xi_{\eta+\eta',\bt}+\sum_{\delta\prec_{wt}\eta+\eta'}R^U_\delta\otimes V(\delta)^*.
$$
It follows that
\begin{equation}\label{degopplex}
(f_\eta\otimes \xi_{\eta,\br})(f_{\eta'}\otimes \xi_{\eta',\br})=(f_\eta f_{\eta'})\otimes \xi_{\eta+\eta',\br+\br'}
\mod R_{<_{alg}(\eta+\eta',\br+\br')},
\end{equation}
which finishes the proof of the lemma.
\end{proof}
\subsubsection{Degeneration}
As a consequence we can define the associated graded algebra
$$
gr R=\bigoplus_{(\la,\bp)\in\Gamma}  R_{\le_{alg} (\la,\bp)}/R_{<_{alg}(\la,\bp)}.
$$
For $Y=G/\hskip -3.5pt/U$ set $A=\mathbb C[G/\hskip -3.5pt/U]$. In this case we have $A_\la^U\simeq\mathbb C$ in \eqref{deg1} for all dominant weights $\la\in\Lambda^+$.
Now Lemma~\ref{sc1} and equation \eqref{degopplex} in the proof of Lemma~\ref{filtR} show:
\begin{coro}\label{grAsemi}
$gr A\simeq \mathbb C[\Gamma]$.
\end{coro}
Let $\mathbb T$ be the torus $(\mathbb C^*)^N$. Recall that we can endow $gr A\simeq \mathbb C[\Gamma]$
naturally with the structure of a $T\times \mathbb T$-algebra, see \eqref{TACTION}.
The description of $gr R$ in general can be reduced to that of $gr A$ as follows: equation~\eqref{degopplex} in the proof of Lemma~\ref{filtR} implies
that $gr R$ may be viewed as the vector space
$$
\bigoplus_{(\la,\bp)\in \Gamma}R^U_{\la^*}\otimes \xi_{\la,\bp}
$$
with componentwise multiplication. Now $T$ acts on $R^U_{\la^*}$ via the character $\la^*$. Recall that we have a second action on $\mathbb C[G]$
coming from the right multiplication of $G$ on $G$. Hence on $\mathbb C[G]^U$ we have a $T$ action
such that $T$ acts on $V(\la)^*\otimes v_\la$ by $\la(t)$ (see \eqref{coordinateGnachU}). If we twist the $T$-action by $w_0:T\rightarrow T$,
$t\mapsto w_0(t)$, then the right action is given by the character $-\la^*$. Since $R$ is finitely generated and without zero divisors, so is $R^U$, and since
$gr A$ is finitely generated, normal (because $\Gamma$ is finitely generated and saturated)
and without zero divisors, so is the ring of invariants $(R^U\otimes gr A)^T$. In particular, this is a normal affine ring,
i.e., the coordinate ring of a normal irreducible affine variety. Summarizing we get (see \cite{AB,Gro}):
\begin{prop}\label{first}\it
The associated graded algebra $gr R$ is, as a $R^G$-$T$-algebra, isomorphic to
$(R^U\otimes gr A)^T$, where $T$ acts on $R^U$ in the standard way and the action on $gr A$ is induced by the twisted right action on $G/\hskip -3.5pt/U$.
As a consequence, the $\mathbb T$-action on $gr A$ makes $gr R$ into a normal affine $T\times \mathbb T$-algebra, and the
$\mathbb T$-invariant subring is isomorphic to $R^G$.
\end{prop}
The proof of the following proposition can be found in \cite{AB}, Proposition 2.2, the only difference being
that we need Lemma~\ref{straight} in section~\ref{speciallinearform} below to go from the multifiltration  to an $\mathbb N$-filtration.
\begin{prop}\label{second}\it
There exists an affine $\mathbb N$-graded $T$-algebra $\mathfrak R$ and a $T$-invariant
element $t\in \mathfrak R_1$ such that
\begin{itemize}
\item[{\it (i)}] $t$ is a nonzero-divisor in $\mathfrak R$, i.e., $\mathfrak R$ is flat over the polynomial ring $\mathbb C[t]$.
\item[{\it (ii)}] The $\mathbb C[t, t^{-1}]$-$T$-algebra $\mathfrak R[t^{-1}]$ is isomorphic to $R[t, t^{-1}]$.
\item[{\it (iii)}] The $T$-algebra $\mathfrak R/t \mathfrak R$ is isomorphic to $gr R$.
\end{itemize}
\end{prop}
\subsection{A special linear form}\label{speciallinearform}
We need a linear form on $\Lambda_{\mathbb R}\times\mathbb R^N$
which satisfies some strict inequalities. Given a finitely generated
monoid $\Gamma\subseteq \Lambda^+\times \mathbb N^N$
and a $\Gamma$-filtration of $\mathbb C[G/\hskip -3.5pt /U]$, the form
is useful to transform the $\Gamma$-filtration into an $\mathbb N$-filtration.	
In this sense, the following lemma plays the same role as Lemma~3.2 in \cite{C1}, it is only adapted to the more general type of
total orders introduced above. We say that a monoid $\Gamma\subseteq \Lambda^+\times \mathbb N^N$
{\it has no exceptional characters\/} 
if $(\la,\mathbf m)\in \Gamma$ and $\mathbf m\not=0$ implies that $\la$ is not a character of $G$.
We use on $\Lambda^+\times \mathbb N^N$ the partial  order ``$>_{alg}$".

\begin{lem}\label{straight}\it
Let $\Gamma\subset \Lambda^+\times \mathbb N^N$ be a finitely generated monoid which has no exceptional characters.
If $M\subset \Gamma$ is a finite subset, then there exists a linear form
$e:\Lambda_{\mathbb R}\times \mathbb R^N \rightarrow  \mathbb R$ such that $e(\Gamma)\subset\mathbb N$,
$e$ takes positive integral values at all $(\la,\mathbf m)\in \Gamma$ such that $\la$ is
not a character of $G$, and $e(\la,\mathbf m)>e(\mu,\mathbf m')$ whenever $(\la,\mathbf m),(\mu,\mathbf m')\in M$ are such that
$(\la,\mathbf m)>_{alg} (\mu,\mathbf m')$.
\end{lem}
\begin{proof}
We consider only the case where ``$>$'' is a $\Psi$-weighted lexicographic or $\Psi$-weighted opposite lexicographic order,
the proof for the right lexicographic orders is completely analogous.
Let $e_1:\Lambda_{\mathbb R}\rightarrow \mathbb R$ be a linear function that takes positive integral values at all positive roots and
non-negative integral values on $\Lambda^+$, such a function always exists. We extend the map $\Psi$ linearly to all of $\mathbb R^N$.

Denote by $M'$ the finite set $\{\mathbf m\mid \exists\la\in\Lambda^+: (\la,\mathbf m)\in M\}$.
By \cite{C1}, Lemma~3.2, there exists a linear map $  e_2:\mathbb N^N\rightarrow \mathbb N$ such that
$  e_2(\mathbf m)>  e_2(\mathbf m')$ whenever $\mathbf m>_{lex} \mathbf m'$.
We claim that we can find
integers $A,B\in \mathbb N$ such that the function $e:\Lambda_{\mathbb R}\times \mathbb R^N \rightarrow  \mathbb R$
defined by
$$\Small
e\,:\,(\la,\mathbf m)\mapsto
\left\{
\begin{array}{rl}
A  e_1(\la)-B  \Psi(\mathbf m)-  e_{2}(\mathbf m)& \hbox{if ``$>$" is a $\Psi$-weighted lexicographic order},\\
A  e_1(\la )-B \Psi(\mathbf m)+  e_{2}(\mathbf m)&  \hbox{if ``$>$" is a $\Psi$-weighted opposite lexicographic order},\\
\end{array}
\right.
$$
has the desired properties above.  For any choice of
$A,B\in\mathbb N_{> 0} $, we have $e(\Lambda\times \mathbb Z^N)\subseteq \mathbb Z$.
Fix a birational sequence $(\zeta^1,\mathbf c^1),\ldots,(\zeta^d,\mathbf c^d)$  of $\Gamma$, fix some $B\in \mathbb N_{> 0} $
and set
$$
A_0=\max\{  e_{2}(\mathbf c^1)+   B\Psi(\mathbf c^1) ,\ldots,  e_{2}(\mathbf c^d)+  B\Psi(\mathbf c^d)  \}+1.
$$
If $(\la,\mathbf m)=\sum_{j=1}^d b_j(\zeta^j,\mathbf c^j)\in \Gamma$, then
$e(\la,\mathbf m)=\sum_{j=1}^d b_j e(\zeta^j,\mathbf c^j)$.
If $ e_{1}(\zeta^j)=0$, then $\zeta^j$ is a character of $G$ and hence $\mathbf c^j=0$.
It follows that $e((\zeta^j,\mathbf c^j))=0$ for any choice of $B$ and  $A$ in this case.
Next suppose $ e_1(\zeta^l)>0$ (and hence  $ e_1(\zeta^l)\ge 1$). If $A\ge A_0$, then
$$
e(\zeta^j,\mathbf c^j)\ge A  e_1(\zeta^j )-B\Psi(\mathbf c^j)-  e_{2}(\mathbf c^j)\ge A_0 - B \Psi(\mathbf c^j)-  e_{2}(\mathbf c^j)\ge 1.
$$
It follows that $e(\Gamma)\subseteq\mathbb N$, and $e(\la,\mathbf m)>0$ if $(\la,\mathbf m)\in \Gamma$ is such that $\la$ is
not a character of $G$.

It remains to show that we can choose $B\in \mathbb N_{> 0} $ and $A\ge A_0$ such that
the inequalities hold.
If $\la=\mu$ and $\Psi(\mathbf m)=\Psi(\mathbf m')$, then $(\la,\mathbf m)>_{alg} (\mu,\mathbf m')$ implies
$\mathbf m<_{lex} \mathbf m'$ in the lexicographic case and $\mathbf m>_{lex} \mathbf m'$ in the opposite lexicographic case.
Note that $e((\la,\mathbf m)-(\la,\mathbf m'))=C e_{2}(\mathbf m- \mathbf m')$,
where $C=\pm 1$, depending on the fixed order. If we are in the lexicographic case,
then $C=-1$, $\mathbf m<_{lex} \mathbf m'$ and hence $e((\la,\mathbf m)-(\la,\mathbf m'))=- e_2(\mathbf m-\mathbf m')>0$.
If we are in the opposite lexicographic case, then $C=1$, $\mathbf m>_{lex} \mathbf m'$ and hence
$e((\la,\mathbf m)-(\la,\mathbf m'))=\bar e_2(\mathbf m-\mathbf m')>0$.

If $(\la,\mathbf m)>_{alg} (\mu,\mathbf m')$, $\la=\mu$ and $\Psi(\mathbf m)\not =\Psi(\mathbf m')$,
then $\Psi(\mathbf m)<\Psi(\mathbf m')$, so $-B(\Psi(\mathbf m)-\Psi(\mathbf m'))>0$.
Since we consider only a finite
number of cases, we can choose $B$ large enough such that in these cases $e((\la,\mathbf m)-(\la,\mathbf m'))=
-B(\Psi(\mathbf m)-\Psi(\mathbf m'))\pm   e_2(\mathbf m-\mathbf m')>0$.

If $\la\not =\mu$, then $(\la,\mathbf m)>_{alg} (\mu,\mathbf m')$ implies $\la \succ_{wt} \mu$ and hence $\la-\mu$ is a sum of positive roots.
It follows that $ e_1(\la-\mu)>0$. Since we consider only a finite number of cases, we can choose $A\ge A_0$ large enough such that in all these
cases $e((\la,\mathbf m)-(\mu,\mathbf m'))=A  e_1(\la-\mu)-B  \Psi(\mathbf m-\mathbf m')\pm e_{2}(\mathbf m-\mathbf m')>0$.
\end{proof}

\subsection{Degenerations in geometric terms}
In geometric terms, Proposition~\ref{second} can be reformulated as follows: The affine $G$-algebra $R$
corresponds to the affine $G$-scheme $Y= \hbox{Spec}(R)$, and $gr R$ corresponds to an affine
$T\times \mathbb T$-scheme denoted by $Y_0$. Proposition~\ref{second} implies:

\begin{coro}\it
There exists a family of affine $T$-schemes
$\rho : \mathcal Y\rightarrow \mathbb A^1$ such that
\begin{enumerate}
\item  $\rho$ is flat,
\item $\rho$ is trivial with fiber $Y$ over the complement of $0$ in $\mathbb A^1$,
\item and
the fiber of $\rho$ at $0$ is isomorphic to $Y_0$.
\end{enumerate}
\end{coro}

We say that $Y$ degenerates to $Y_0$, and $Y_0$ is a limit of $Y$.

\subsection{Spherical varieties}
Next we consider the special case that $Y$ is an affine spherical $G$-variety $Y = \hbox{Spec}(R)$,
i.e., $Y$ is a normal affine $G$-variety and contains a dense $B$-orbit. Note that
the multiplicities for the isotypic decomposition are smaller than or equal to one. One can associate
to $Y$ a rational polyhedral convex cone $\hbox{Cone}(Y)\subseteq \Lambda\otimes {\mathbb R}$,
called the {\it weight cone} of $Y$.  The cone is uniquely determined by $Y$, and
$$
R\simeq \bigoplus_{\la\in \Lambda^+\cap \hbox{\tiny Cone}(Y)} V(\la)
$$
as $G$-module. By assumption, $\Gamma$ is finitely generated and saturated, denote by ${\cal C}_{(S,>)}\subset \Lambda_{\mathbb R}\times
\mathbb R^N$
the polyhedral cone spanned by $\Gamma$.
Now Propositions~\ref{first} and \ref{second} imply in this special case:
\begin{prop}\it
Let $Y$ be an affine spherical $G$-variety with weight cone $\hbox{Cone}(Y)$.
Then $Y$ degenerates to the affine toric $T \times \mathbb T$-variety $Y_0$
such that
$$
\hbox{\rm Cone}(Y_0) = (\hbox{\rm Cone}(Y) \times \mathbb R^N) \cap {\cal C}_{(S,>)}.
$$
If $S=(\beta_1,\ldots,\beta_N)$, then
the degeneration is $T$-equivariant, where $T$ acts on $Y_0$ via the homomorphism
$T \hookrightarrow T \times \mathbb T$, $t \mapsto (t^{-1},(\beta_1(t))^{-1}, \ldots, (\beta_N(t))^{-1})$.
\end{prop}
Returning to an arbitrary $G$-scheme $Y$ with limit $Y_0$, many geometric properties
hold for $Y$ if and only if they hold for $Y_0$, see \cite{AB}, or \cite{Pop} and \cite{Gro}, \S 18.
The main ingredient in the proof is the isomorphism between $gr R$ and $(R^U\otimes gr A)^T$.
\begin{prop}\label{normal}\it
Let $Y$ be an affine $G$-variety, with degeneration $Y_0$ corresponding
to the choice of an $(\mathbb N^N,>,S)$-filtration on $\mathcal{U}(\mathfrak n^-)$, where $(S,>)$ is a sequence with an ASM.
Then $Y$ is normal if and only if $Y_0$ is normal.
\end{prop}

\subsection{Degeneration of polarized projective $G$-varieties}
We follow again the approach of Alexeev and Brion \cite{AB}.
Let $(Y,\cal L)$ be a  polarized projective $G$-variety, i.e. $Y$ is a normal projective $G$-variety together with an
ample $G$-linearized  invertible sheaf $\cal L$.

The sheaf ${\cal L}^n := {\cal L}^{\otimes n}$ is also $G$-linearized, so the space
of global sections $H^0(Y,{\cal L}^n)$ is a finite-dimensional rational $G$-module.
Consider the associated graded algebra
$$
R(Y,{\cal L}):= \bigoplus_{n\in\mathbb N} H^0(Y,{\cal L}^n),
$$
this is a finitely generated, integrally closed domain. We have $Y = \hbox{Proj\,} R(Y,{\cal L})$ and
${\cal L}^n={\cal O}_Y(n)$.
Note that ${\cal L}$ is globally generated whenever $Y$ is a spherical variety.

We endow the algebra $R(Y,{\cal L})$ with a filtration as in the section before.
If $Y$ is projectively normal, then we can view $R(Y,{\cal L})$ as the coordinate ring of the affine
cone over the embedded variety $Y\hookrightarrow \mathbb P(V)$. If $Y$ is projectively normal and a spherical variety,
then the affine cone over $Y$ is an affine spherical variety for the group $\mathbb C^*\times G$.

The associated algebra $gr R(Y,{\cal L})$ is still finitely generated, it is an
integrally closed domain, with an action of $\mathbb C^*\times T \times \mathbb T$
such that the $\mathbb C^*$-action defines a positive grading.
The projective variety $Y_0 :=\hbox{Proj\,} gr R(X,\cal L)$ is again
a projective $T\times\mathbb  T$-variety equipped with $T\times\mathbb  T$-linearized sheaves
${\cal L}_0^{(n)}={\cal O}_{Y_0}(n D)$ for all integers $n$, where $D$ refers to a $\mathbb Q$-Weil divisor.
Further, $D$ is $\mathbb Q$-Cartier and ample, i.e., the sheaf ${\cal L}_0^{(m)}$
 is invertible and ample for any sufficiently divisible integer $m > 0$.
In particular, every sheaf ${\cal L}_0^{(n)}$ is divisorial, i.e., it is the sheaf of sections of an integral Weil divisor.

We just quote \cite{AB}, the proof is the same:
\begin{thm}\it
Let $(Y,\cal L)$ be a  polarized $G$-variety and let
$(S,>)$ be a sequence with an ASM. Consider the
induced $(\mathbb N^N,>,S)$-{\it filtration} on $R(Y,{\cal L})$.
Then there exists a family of $T$-varieties $\pi : {\cal Y}\rightarrow \mathbb A^1$, where ${\cal Y}$ is a
normal variety, together with divisorial sheaves ${\cal O}_{\cal Y}(n)$ ($n\in\mathbb  Z$), such that
\begin{itemize}
\item[{\it i)}] $\pi$ is projective and flat.
\item[{\it ii)}] $\pi$ is trivial with fiber $Y$ over the complement of $0$ in $\mathbb A^1$, and ${\cal O}_{\cal Y}(n)\vert_Y \simeq {\cal L}^n$
for all $n$.
\item[{\it iii)}]  The fiber of $\pi$ at $0$ is isomorphic to $Y_0$, and ${\cal O}_{\cal Y}(n)\vert_{Y_0} \simeq {\cal L}_0^{(n)}$ for all $n$.
\end{itemize}
In addition, if $Y$ is spherical, then $Y_0$ is a toric variety for the torus $T\times\mathbb T$.
\end{thm}
\section{Moment polytopes and Newton-Okounkov bodies}\label{polytopesection}
\subsection{Moment polytopes}
The results and definitions in this section on moment polytopes can be found in \cite{AB}, the only
difference being that we consider not only the string cone case.
Let $(Y,\cal L)$ be a  polarized spherical $G$-variety and let
$(S,>)$ be a sequence with an ASM,
so $\Gamma$ is assumed to be finitely generated and saturated.
We recall the definition of a moment polytope of the polarized $G$-variety $(Y,\cal L)$.
Note that for a dominant weight $\lambda$, the isotypical component $H^0(Y ,{\cal L}^{n})_{(\lambda)} \not= 0$
if and only if the space of $U$ invariants of weight $\lambda$ is not trivial: $H^0(Y,{\cal L}^n)^U_\lambda \not=0$.
Further, the algebra $R(Y,{\cal L})^U$ is finitely generated and $\Lambda^+\times\mathbb N$-graded;
let $(f_i)_{i=1,\ldots,r}$ be homogeneous
generators and $(\lambda_i, n_i)$ their weights and degrees.

\begin{defn}
The convex hull  of the points $\frac{\lambda_i}{n_i}$, $i=1,\ldots,r$, in $\Lambda_{\mathbb R}$
is called the {\it moment polytope $P(Y,\cal L)$} of the polarized $G$-variety $(Y,\cal L)$.
\end{defn}
Another way to view the moment polytope: the points $\frac{\lambda}{n}\in\Lambda_{\mathbb R}$
such that $\lambda\in\Lambda^+$, $n\in\mathbb N_{>0}$, and the isotypical component $H^0(Y,{\cal L}^n)_{(\lambda)}$
is nonzero, are exactly the rational points of the rational convex polytope $P(Y,{\cal L})\subset \Lambda_{\mathbb R}$.
Further,  $P(Y,{\cal L}^m) = mP(Y,{\cal L})$ for any positive integer $m$.

By positive homogeneity of the moment polytope, this definition extends to
$\mathbb Q$-polarized varieties, in particular, to any limit $(Y_0,{\cal L}_0)$ of $(Y,{\cal L})$.
We denote by $P(\Gamma,Y,{\cal L})$ the moment polytope of that limit. It is a rational convex polytope in
$\Lambda_{\mathbb R}\times \mathbb R^N$, related to the moment polytope of $(Y,{\cal L})$ by the following theorem:

\begin{thm}\label{polytopethm}\it
The projection $p : \Lambda_{\mathbb R}\times \mathbb R^N \rightarrow \Lambda_{\mathbb R}$ onto the first factor
restricts to a surjective map
$$
p : P(\Gamma,Y,{\cal L}) \rightarrow P(Y,{\cal L}),
$$
with fiber over any $\lambda\in\Lambda^+_{\mathbb R}$
being the essential polytope $P(\lambda)$ (see \ref{esspoly}). In particular, for $\lambda\in\Lambda^+$, let
$P_\lambda$ be the stabilizer  of the line $[v_\lambda]\in \mathbb P(V(\lambda))$.
The limit of the flag variety $G/P_\lambda$ is a toric variety under $\mathbb T$, and its
moment polytope is the essential polytope $P(\lambda)$.
\end{thm}
\begin{proof} By definition, the rational points of $P(\Gamma,Y,{\cal L})$ are the pairs
$(\frac{\lambda}{n}, \frac{\mathbf m}{n})$ such that $H^0(Y,{\cal L}^n)_{(\lambda)}\not=0$
and $(\lambda^*, \mathbf m) \in {\cal C}_{(S,>)}$. This implies the first assertion for rational
$\lambda$, and hence for all $\lambda$ since all involved polytopes are rational.
Consider the $G$-variety $Y = G/P_\lambda$  and the $G$-linearized line bundle ${\cal L} = {\cal L}_\lambda$.
Then ${\cal L}$ is ample, and $R(Y,{\cal L}) =\bigoplus_{n=0}^\infty V(n\lambda^*)$. Thus, $Y$ is spherical with
moment polytope being the point $\lambda^*$.
This implies the second assertion.
 \end{proof}
\subsection{Newton-Okounkov bodies} Let us first assume that $\Gamma$ is not necessarily finitely generated, so we just have a birational
sequence $S$ and a $(\mathbb N^N,>,S)$-{\it filtration} together with the monoid $\Gamma$, which we can view
either as the global essential monoid (Corollary~\ref{semigroup}) or as the valuation monoid ${\cal V}(G/\hskip -3.5pt/U)$
(Proposition~\ref{semigleichsemi1}). Recall that the valuation $\nu$ is determined by the valuation $\nu_1$ (see \eqref{valuationdef1} and
Remark~\ref{semigleichsemi1}). The valuation $\nu_1$ depends on the choice of $S$  and ``$>$'', and the birational map
$\pi:Z_S\rightarrow U^-$ (see \eqref{birational1}) provides a birational map $\pi:Z_S\rightarrow G/B$ by identifying $U^-$ with an affine neighborhood
of the identity. Hence we can view $\nu_1$ naturally as a $\mathbb Z^N$-valued valuation on $\mathbb C(G/B)$.

We leave the general case of spherical varieties to the reader and consider only flag varieties.
Let $\lambda$ be a regular dominant weight (otherwise replace $G/B$ by $G/Q$, where $Q$ is a parabolic subgroup such that
the corresponding line bundle is very
ample on $G/Q$) and let $\mathcal L_\la$
be the corresponding very ample line bundle on $G/B$. Let $R_\la$ be the ring $\bigoplus_{n\ge 0} H^0(G/B,\mathcal L_{n\la})$.
Recall that $H^0(G/B,\mathcal L_{n\la})\simeq V(n\la)^*$ as a $G$-representation, we fix for all $n\in\mathbb N$ the
dual vector $\xi_{n\la,\mathbb O}$ to the fixed highest weight vector $v_{n\la}\in V(n\la)$. The  Newton-Okounkov body
associated to $G/B$ depends on the choice of the valuation $\nu_1$, the ample line bundle $\mathcal L_\la$ and the
choice of a non-zero element in $H^0(G/B,\mathcal L_{\la})$, in our case the vector $\xi_{n\la,\mathbb O}$.
The {\it Newton-Okounkov body} $\Delta_{\nu_1}(\la)$ is defined as follows.
(For more details on Newton-Okounkov bodies see for example \cite{KK,K1}). One associates to the graded ring $R_\la$ the monoid
$$
\bigcup_{n>0}\{(n,\nu_1(\frac{s}{\xi_{n\la,\mathbb O}}) \mid s\in H^0(G/B,\mathcal L_{n\la})\}.
$$
In view of Proposition~\ref{semigleichsemi1}, this is nothing but the essential monoid $\Gamma(\la)$ associated to $\la$ (Corollary~\ref{semigroup}).
\begin{defn}
The {\it Newton-Okounkov body} $\Delta_{\nu_1}(\la)$ is the convex closure
$$\text{conv}\overline{(\bigcup_{n\in \mathbb N}\{\frac{\mathbf m}{n}\mid   \mathbf m\in es(n\la)\}}).$$
\end{defn}
It follows immediately:
\begin{prop}\it
Let $S$ be a birational sequence and fix a $\Psi$-weighted lex order. If the
global essential monoid $\Gamma(S,>)$ is finitely generated, then the Newton-Okounkov body
$\Delta_{\nu_1}(\la)$ is equal to the essential polytope $P(\la)$.
\end{prop}
\begin{exam}\label{kiritchenkoexample}
We assume $G=SL_{n+1}$. Consider the birational sequence:
$$
S=(\alpha_{n,n}, \alpha_{n-1,n}, \alpha_{n-1,n-1},\ldots,\alpha_{1,n}, \ldots,\alpha_{1,1}),
$$
which is of PBW-type and is associated to the following reduced decomposition: $\underline{w}_0=s_n(s_{n-1}s_n)\cdots(s_{1}\cdots s_n)$.
If we identify $U^-$ with the open affine neighborhood of the identity in $G/B$, then the choice of $S$ provides a natural
sequence of affine subvarieties  given by the products of just the terminal $N-1$ respectively terminal $N-2$ etc. root subgroups. This
sequence is compatible with the sequence of translated Schubert varieties (we omit the translation by $w_0$)
used in \cite{Kir2}. If we choose the lexicographic ordering on $\mathbb N^N$, then the associated valuation monoid
is just the one considered in \cite{Kir2}. Note that the valuation monoid or, equivalently, the global essential monoid,
does not change if we replace the lexicographic ordering by the height weighted lexicographic ordering, see Remark~\ref{changepsi}.
Now in \cite{FFL2} we determine the essential monoid for the sequence $S$ above, using the homogeneous lexicographic order. Now in view
of the results of the next section (Example~\ref{leftright}, Proposition~\ref{ordering1}) and the fact that all
relations used in  \cite{FFL2} are homogeneous, one can show that the proofs go through in the same way
if one replaces the homogeneous lexicographic order by the height weighted lexicographic order. So this explains
why one gets in \cite{FFL2} as well as in \cite{Kir2} the FFLV-polytopes as essential polytopes respectively as
Newton-Okounkov bodies.
\end{exam}
\begin{rem} We will remark briefly on birational sequences of subgroups of $U^{-}$, toric degenerations of Schubert varieties and how the induced cones are related to the cones for $U^{-}$ and the flag variety.
\par
Let $w$ be a Weyl group element, $\Phi_w = w^{-1} (\Phi^+) \cap \Phi^-$ and $U_w^-$ be the subgroup of $U^-$ generated by the corresponding one-dimensional root subgroups. Let $S_w$ be a birational sequence for $U_w^-$ (similarly defined using the torus $\mathbb{T}^{\operatorname{dim } U_w^-} \times T$, where $T$ is the fixed torus of $G$). Let $\Psi_w$ be a weight function on 
 $S_w$ and let $<_{w}$ be a $\Psi_w$-weighted lexicographic order.
\par
Similarly to the construction in Sections~\ref{essential1} and \ref{filtandsemi}, one can use 
the Demazure module $V(\lambda)_w\subset V(\lambda)$ associate to $w$ to construct the essential monoid
$\Gamma_w$ and the essential cone $\mathcal{C}_w$. Under the finitely generation and saturation conditions, 
toric degenerations for the Schubert variety (corresponding to $w$)   can be constructed.
\par
Any birational sequence $S_w$ for $U_w^-$ can be extended to a birational sequence $S$ of  $U^{-}$. One can also 
extend the function $\Psi_w$ and the order $<_w$ to provide a $\Psi$-weighted lexicographic order $<$ on $S$, 
 leading to a monoid $\Gamma$ and an essential cone $\mathcal{C}$. These extensions are not unique, 
so any relation between $\Gamma_w$ and $\Gamma$ (resp. $\mathcal{C}_w$ and $\mathcal{C}$) depend on the chosen extensions. 
\par
Some special cases of such extensions have been already studied.  Let ${w} \in W$ be an element in the Weyl group. We fix a reduced expression 
$\underline{w}=s_{i_1}\cdots s_{i_r}$ and choose correspondingly
the birational sequence for $U_w^-$, $\Psi_w$ and $<_w$ as in the \textit{reduced expression case} (Section~\ref{stringconeproof}). 
Denote by $\mathcal{C}_{\underline{w}}$ the corresponding essential cone. Extend now the reduced expression above
to a reduced expression of $w_0$ at the \textit{right end}. It has been shown in \cite{C1} that $\mathcal{C}_{\underline{w}}$ is a face of 
$\mathcal{C}_{\underline{w}_0}$. A similar result has been shown in \cite{F} for $G=SL_{n+1}$ and \textit{triangular} Weyl group elements in the homogeneous 
PBW-case (Section~\ref{homordercase}). 
\end{rem}
\section{Quasi-commutative filtrations in the PBW-type case and ${\mathbb G}^N_a$-actions}\label{Qcf}
Let $\Phi^+=\{\beta_1,\ldots,\beta_N\}$ be an enumeration of the set of positive roots,
so $S=(\beta_1,\ldots,\beta_N)$ is a birational sequence and $\{f_{\beta_1},\ldots,f_{\beta_N}\}$ is a basis for $\mathfrak n^-$.
Fix a  $(\mathbb N^N,>,S)$-{filtration} on $\mathcal{U}(\mathfrak n^-)$.
If a product of root vectors $m=f_{{i_1}}^{\ell_1}\cdots f_{{i_t}}^{\ell_t}$
is not necessarily ordered, then let ${\bf exp}(m)=(k_1,\ldots,k_N)$ be the ordered sequence of the exponents, i.e., $k_1=\sum_{i_j=1}\ell_j$ is the sum of all the exponents of
the various times $f_{\beta_1}$ occurs in the product, $k_2=\sum_{i_j=2}\ell_j$ is the sum of all the exponents of
the various times $f_{\beta_2}$ occurs in the product etc.
The monomial ${\bf f}^{{\bf exp}(m)}$ is  called the {\it ordered monomial} associated to a monomial $m$.
\begin{defn}\label{quasicomdefn}
The  $(\mathbb N^N,>,S)$-{\it filtration} of $\mathcal{U}(\mathfrak n^-)$ is called {\it quasi-commutative} if for all monomials $m$ in elements of $S$
the following rule holds:
\begin{equation}\label{straightening}
m=f_{{i_1}}^{\ell_1}\cdots f_{{i_t}}^{\ell_t}= {\bf f}^{{\bf exp}(m)} + \sum_{\mathbf n<{\bf exp}(m)} a_{\mathbf n} {\bf f}^{{\mathbf n}}.
\end{equation}
\end{defn}
\begin{exam}
Let $S=(\beta_1,\ldots,\beta_N)$ be a birational sequence.
If $\Psi$ satisfies $\Psi(e_i)+\Psi(e_j)>\Psi(e_k)$ whenever $\beta_i+\beta_j=\beta_k$, then the
$(\mathbb N^N,>,S)$-{\it filtration} of $\mathcal{U}(\mathfrak n^-)$ is quasi-commutative.
The function $\Psi$ used in section~\ref{homordercase} has this property as well as the function used in
Example~\ref{quantumex}. The cases considered in \cite{BD} can also be reformulated in this setting,
leading to a semigroup which is not isomorphic to the one in section~\ref{homordercase}.
\end{exam}
Let $V(\la)$ be an irreducible representation with highest weight vector $v_\lambda$.
The definition above implies in the quasi-commutative case:
\begin{equation}\label{dualcalc}
 f_{\beta_i} \mathbf f^{(\mathbf m)}v_\la=
 (m_i+1)\mathbf f^{(\mathbf m+\mathbf e_i)}v_\la +\sum_{\mathbf k< \mathbf m+\mathbf e_i} a_{\mathbf k}\mathbf f^{(\mathbf k)}v_\la,
\end{equation}
and hence $ f_{\beta_i} V(\la)_{\le \mathbf m}\subseteq V(\la)_{\le \mathbf m+\mathbf e_i}$. It follows immediately:
\begin{prop}\it
Suppose the $(\mathbb N^N,>,S)$-{\it filtration}  of $\mathcal{U}(\mathfrak n^-)$ is quasi-commutative.
\begin{itemize}
\item[{\it i)}] The multiplication on $\mathcal{U}(\mathfrak n^-)$ induces on $\mathcal{U}^{gr}(\mathfrak n^-)$ the structure of
a commutative algebra such that the natural map $\mathcal{U}^{gr}(\mathfrak n^-)\rightarrow S(\mathfrak n^-)$
given by $\bar{\bf f}^{\bf n}\mapsto {\bf f}^{\bf n}$ defines an isomorphism of commutative algebras.
\item[{\it ii)}] The action of  $\mathcal{U}(\mathfrak n^-)$ on $V(\la)$ makes $V^{gr}(\la)$ into a cyclic $S(\mathfrak n^-)$-module
generated by the class $\bar v_\la$ of $v_\la$ in $V^{gr}(\la)$, such that the annihilating ideal of $\bar v_\la$ is the monomial ideal
having as $\mathbb C$-basis $\{ \mathbf f^{(\mathbf m)} \mid  {\bf m}\in es(\mathfrak n^-), {\bf m}\not\in es(\la)\}$,
and the vectors $\{\mathbf f^{(\mathbf m)}\bar v_\la\mid {\bf m}\in es(\la)\}$ form a basis for $V^{gr}(\la)$.
\item[{\it iii)}] The element $f_{\beta_i}$ acts homogeneously of degree ${\mathbf e}_i$ on $V^{gr}(\la)$, i.e., if $\mathbf m,\mathbf m+\mathbf e_i\in es(\la)$,
then $f_{\beta_i}\mathbf f^{(\mathbf m)}\bar v_\la=(m_i+1)\mathbf f^{(\mathbf m+\mathbf e_i)}\bar v_\la$ in  $V^{gr}(\la)$.
\end{itemize}
\end{prop}
The action of the commuting nilpotent Lie algebras $\mathbb Cf_{\beta_j}\subset S(\mathfrak n^-)$, $j=1,\ldots,N$,
on $V^{gr}(\la)$ can be integrated into an action of the unipotent group $\mathbb G_a^N$:
\begin{coro}\it
The action of $S(\mathfrak n^-)$ on $V^{gr}(\la)$ induces an action of $\mathbb G_a^N$ on $V^{gr}(\la)$.
\end{coro}
Before we come to the dual situation, we point out the following consequence of Definition~\ref{quasicomdefn}:
\begin{lem}\label{deriving}\it
If $\mathbf m\not\in es(\la)$, then $\mathbf{m+e}_i\not\in es(\la)$ for all $i=1,\ldots,N$.
\end{lem}
\begin{proof}
If $\mathbf m\not\in es(\la)$, then $\mathbf f^{\mathbf m}v_\la=\sum_{\mathbf k<\mathbf m} a_{\mathbf k} \mathbf f^{\mathbf k}v_\la$ and hence
$f_{\beta_i} \mathbf f^{\mathbf m}v_\la=\sum_{\mathbf k<\mathbf m} a_{\mathbf k} f_{\beta_i}\mathbf f^{\mathbf k}v_\la$. Now \eqref{dualcalc} implies
that $f_{\beta_i}\mathbf f^{\mathbf k}v_\la=\sum_{\mathbf p\le \mathbf{k+e}_i} b_{\mathbf p}\mathbf f^{\mathbf p}v_\la$, so
$f_{\beta_i} \mathbf f^{\mathbf m}v_\la$ is a linear combination of vectors of the form $\mathbf f^{\mathbf p}v_\la$ such that
$\mathbf p<\mathbf{m+e}_i$ because $\mathbf k<\mathbf m$ implies $\mathbf{k+e}_i<\mathbf{m+e}_i$.
Applying \eqref{dualcalc} to $f_{\beta_i}\mathbf f^{\mathbf m}v_\la$ implies hence that $\mathbf f^{\mathbf{m+e}_i}v_\la$
is a linear combination of vectors $\mathbf f^{\mathbf p}v_\la$, $\mathbf p<\mathbf{m+e}_i$, so $\mathbf{m+e}_i\not\in es(\la)$.
\end{proof}
Recall that we turn around the filtration in the dual situation, for $\mathbf m\in es(\la)$ set (see Lemma~\ref{<1a}):
$$
V^*(\la)_{\ge \mathbf m}=\langle \xi_{\la,\mathbf k}\mid \mathbf k\in es(\la),\mathbf k\ge \mathbf m\rangle=V(\la)_{< \mathbf m}^\perp,\quad
$$
$$
V^*(\la)_{>\mathbf m}=\langle \xi_{\la,\mathbf k}\mid  \mathbf k\in es(\la), \mathbf k\ge \mathbf m \rangle=V(\la)_{\le \mathbf m}^\perp.
$$
{
For $\mathbf{n,p}\in es(\la)$ we have by \eqref{dualcalc}:
\begin{equation}\label{dualcalculation2}
( f_{\beta_i}\xi_{\la,\mathbf p}) (\mathbf f^{(\mathbf n)}v_\la)=
\xi_{\la,\mathbf p}\big(-(n_i+1)\mathbf f^{(\mathbf n+\mathbf e_i)}v_\la -
\sum_{\mathbf k< \mathbf n+\mathbf e_i} a_{\mathbf k}\mathbf f^{(\mathbf k)}v_\la\big) \not=0
\Rightarrow \mathbf p\le \mathbf n+\mathbf e_i.
\end{equation}
It follows that $f_{\beta_i}\xi_{\la,\mathbf p}$ is a linear combination of $\xi_{\la,\mathbf n}$ such that $\mathbf n\in es(\la)$ and
$\mathbf p\le \mathbf n+\mathbf e_i$ and hence:
$$
f_{\beta_i}\xi_{\la,\mathbf p}\in \sum_{\substack{\mathbf n\in es(\la)\\ \mathbf p\le \mathbf n+\mathbf e_i}}V^*(\la)_{\ge \mathbf n}.
$$
If $\mathbf q\in es(\la)$ is such that $\mathbf q>\mathbf p$, then $f_{\beta_i}\xi_{\la,\mathbf q}$ is a linear combination of $\xi_{\la,\mathbf n}$ such that
$\mathbf p<\mathbf q\le \mathbf n+\mathbf e_i$ and $\mathbf n\in es(\la)$, so we get an induced map:
\begin{equation}\label{dualcalculation2}
f_{\beta_i}: \frac{V^*(\la)_{\ge \mathbf p}}{V^*(\la)_{>\mathbf p}}
\rightarrow \frac{ \sum_{\substack{\mathbf n\in es(\la)\\ \mathbf p\le \mathbf n+\mathbf e_i}}V^*(\la)_{\ge \mathbf n}}
{ \sum_{\substack{\mathbf n\in es(\la)\\ \mathbf p < \mathbf n+\mathbf e_i}}V^*(\la)_{\ge \mathbf n}}
\simeq
\left\{\begin{array}{rl}0&\text{if $\not\exists \mathbf n$ such that $\mathbf p=\mathbf n+\mathbf e_i$},\\
\frac{V^*(\la)_{\ge {\mathbf p-\mathbf e_i}}}{V^*(\la)_{>\mathbf p-\mathbf e_i}}&\text{if $\exists \mathbf n$ such that $\mathbf p=\mathbf n+\mathbf e_i$}.
\end{array}\right.
\end{equation}
The equation \eqref{dualcalculation2} and Lemma~\ref{deriving} imply more precisely:
\begin{prop}\label{dualoperator}\it
Suppose the $(\mathbb N^N,>,S)$-{\it filtration}  of $\mathcal{U}(\mathfrak n^-)$ is quasi-commutative.
The action of  $\mathcal{U}(\mathfrak n^-)$ on $V^*(\la)$ makes $V^{gr,*}(\la)$ into a $S(\mathfrak n^-)$-module.
The element $f_{\beta_i}$ acts homogeneously on $V^{gr}(\la)$ of degree $-{\bf e}_i$, i.e., if $\mathbf m+\mathbf e_i\in es(\la)$,
then $\mathbf m\in es(\la)$ and $f_{\beta_i}(\xi_{\la,\mathbf m+\mathbf e_i})=(m_i+1)\mathbf\xi_{\la,\mathbf m}$ in  $V^{gr,*}(\la)$, and
$f_{\beta_i}(\xi_{\la,\mathbf p})=0$ in $V^{gr,*}(\la)$ for $\mathbf p\in es(\la)$ if $\mathbf p\not= \mathbf n +\mathbf e_i$ for all $\mathbf n\in es(\la)$.
\end{prop}
}
\begin{coro}\it
The action of $S(\mathfrak n^-)$ on $V^{gr,*}(\la)$ can be integrated into an action of $\mathbb G_a^N$ on $V^{gr,*}(\la)$.
\end{coro}
Let $Y$ be an irreducible affine $G$-variety with coordinate ring $R=\mathbb C[Y]$. The filtration of $R$
in section~\ref{algfilt} is compatible with the construction above, so we get a representation of $S(\mathfrak n^-)$ on $gr R$.
\begin{lem}\it
If the $(\mathbb N^N,>,S)$-{\it filtration}  of $\mathcal{U}(\mathfrak n^-)$ is quasi-commutative, then
$S(\mathfrak n^-)$ acts on $R$ by derivations.
\end{lem}
\begin{proof}
If $\mathbf m+\mathbf e_i\in es(\la),\mathbf k+\mathbf e_i\in es(\mu)$, then $\mathbf m\in es(\la)$ and $\mathbf k\in es(\mu)$
by Lemma~\ref{deriving}. It remains to show that the Leibniz rule holds.
{
We have
$$
f_{\beta_i}(\xi_{\la,\mathbf m+\mathbf e_i}\xi_{\mu,\mathbf k+\mathbf e_i})=f_{\beta_i}(\xi_{\la+\mu,\mathbf{m+k+} 2\bold{e}_i})=(m_i+k_i+ 2)\xi_{\la+\mu,\mathbf{m+k+e}_i},
$$
and the following calculation shows that the rule holds in this case:
$$
\begin{array}{rcl}
& & \big(f_{\beta_i}(\xi_{\la,\mathbf m+\mathbf e_i})\big)\xi_{\mu,\mathbf k+\mathbf e_i}+
\xi_{\la,\mathbf m+\mathbf e_i}\big(f_{\beta_i}(\xi_{\mu,\mathbf k+\mathbf e_i})\big)
\\&=&(m_i+1)(\xi_{\la,\mathbf m}\xi_{\mu,\mathbf k+\mathbf{e}_i})+(k_i+1)(\xi_{\la,\mathbf m+\mathbf{e}_i}\xi_{\mu,\mathbf k})\\
&=&(m_i+k_i+ 2)\xi_{\la+\mu,\mathbf{m+k+e}_i}.
\end{array}
$$
Next assume that $\mathbf m'\in es(\la)$ is such that $\mathbf m'\not=\mathbf{m+e}_i$ for all $\mathbf{m}\in es(\la)$.
By Lemma~\ref{deriving}, this is only possible if $m'_i=0$. Hence we get for $\mathbf k+\mathbf e_i\in es(\mu)$ by Proposition~\ref{dualoperator}:
$$
f_{\beta_i}(\xi_{\la,\mathbf m'}\xi_{\mu,\mathbf k+\mathbf e_i})=f_{\beta_i}(\xi_{\la+\mu,\mathbf{m'+k+} \bold{e}_i})=(k_i+ 1)\xi_{\la+\mu,\mathbf{m'+k+e}_i},
$$
and
$$
\begin{array}{rcl}
\big(f_{\beta_i}(\xi_{\la,\mathbf m'})\big)\xi_{\mu,\mathbf k+\mathbf e_i}+
\xi_{\la,\mathbf m'}\big(f_{\beta_i}(\xi_{\mu,\mathbf k+\mathbf e_i})\big)=(k_i+1)(\xi_{\la,\mathbf m'}\xi_{\mu,\mathbf k})
&=& (k_i+1)\xi_{\la+\mu,\mathbf{m'+k}}\\
&=&f_{\beta_i}(\xi_{\la+\mu,\mathbf{m'+k+e}_i}).
\end{array}
$$
Finally, if neither $\mathbf m'\in es(\la)$ nor $\mathbf k'\in es(\mu)$ can be written as $\mathbf{m}+\bold{e}_i$ respectively
$\mathbf{k+e}_i$, then $m'_i=k'_i=0$, and hence $\big(f_{\beta_i}(\xi_{\la,\mathbf m'})\big)\xi_{\mu,\mathbf k'}+
\xi_{\la,\mathbf m'}\big(f_{\beta_i}(\xi_{\mu,\mathbf k'})\big)=0$ as well as $f_{\beta_i}(\xi_{\la+\mu,\mathbf{m'+k'}})=0$.
}
\end{proof}
As an immediate consequence we get:
\begin{thm}\it
Let $Y$ be an irreducible affine $G$-variety with coordinate ring $R=\mathbb C[Y]$.
If the $(\mathbb N^N,>,S)$-{\it filtration}  of $\mathcal{U}(\mathfrak n^-)$ is quasi-commutative, then
$\mathbb G_a^N$ acts on $gr R$ by algebra isomorphisms. In particular, if $(S,>)$ is a sequence with an ASM,
then we get an induced $\mathbb G_a^N$-action on the limit variety $Y_0$.
\end{thm}
\begin{exam}\label{homcommute}
If the order ``$>$'' is a homogeneous order as in Example~\ref{homorder}, then reordering a monomial
gives only additional terms of strictly lower degree. So with respect to this ordering, these additional
terms are strictly smaller than ${\bf exp}(m)$ and hence the  $(\mathbb N^N,>,S)$-{\it filtration} on $\mathcal{U}(\mathfrak n^-)$ is {quasi-commutative}.
\end{exam}
We will see in Proposition~\ref{ordering1} that another interesting class of examples for quasi-commutative filtrations can be obtained
by imposing on $S$ the following condition:
\begin{defn}
The sequence $S$ satisfies the {\it (left) bracket condition} if for all  $i<j$ there exists a $k>i$ and $a_{i,j,k}\in\mathbb C$ such that:
$[f_i,f_j]=a_{i,j,k}f_k$.
$S$ satisfies the {\it right bracket condition} if for all  $i<j$ there exists a $k<j$ and $a_{i,j,k}\in\mathbb C$ such that $[f_i,f_j]=a_{i,j,k}f_k$.
\end{defn}
\vskip 3pt
\begin{exam}\label{leftright}
Fix a reduced decomposition of the longest word in the Weyl group $w_0=s_{i_1}\cdots s_{i_N}$
and let $\beta_1=\alpha_{i_1}$, $\beta_2=s_{i_1}(\alpha_{i_2})$, $\beta_3=s_{i_1}s_{i_2}(\alpha_{i_3})$
etc.  Then $S=(f_{\beta_1},\ldots,f_{\beta_N})$ satisfies the left and the right bracket condition because if $i<j$ and
$[f_{\beta_i},f_{\beta_j}]=cf_{\beta_k}$ for some nonzero constant $c$, then
$\beta_i+\beta_j=\beta_k$ and hence $i<k<j$.
\end{exam}
\begin{exam}
Fix a total order on $\Phi^+$: $\beta_1>\ldots>\beta_N$ such that if $\beta_k\succ \beta_i$ in the usual partial
order on roots, then $k>i$. The sequence $S=(\beta_1,\ldots,\beta_N)$ satisfies the {\it left bracket condition}.
\end{exam}
\begin{exam}
Fix a total order on $\Phi^+$: $\beta_1>\ldots>\beta_N$ such that if $\beta_k\succ \beta_j$ in the usual partial
order on roots, then $k<j$. The sequence $S=(\beta_1,\ldots,\beta_N)$ satisfies the {\it right bracket condition}.
(An ordering with such a property is called a {\it good ordering} in \cite{FFL1}).
\end{exam}
\begin{exam}
Let $\mathfrak{g}$ be of type $\tt C_n$. In \cite{FFL3}, Section 2, the following sequence is used:
$$S=(\alpha_{n,n},\alpha_{n-1,\overline{n-1}},\alpha_{n-1,n},\alpha_{n-1,n-1},\alpha_{n-2,\overline{n-2}},\alpha_{n-2,\overline{n-1}},\ldots \alpha_{n-2,n-2},\ldots, $$
$$\alpha_{1,\overline{1}},\alpha_{1,\overline{w2}},\ldots,\alpha_{1,\overline{n-1}},\alpha_{1,n},\ldots,\alpha_{1,2},\alpha_{1,1}).$$
Since the total order on positive roots corresponding to this sequence is not convex, it does not arise from any reduced expression
$\underline{w}_0$, but the right bracket condition is satisfied.
\end{exam}
\begin{exam}
Let $\mathfrak{g}$ be of type $\tt G_2$. In \cite{Gor1}, the following sequence is studied:
$$S=(3\alpha+2\beta,3\alpha+\beta,2\alpha+\beta,\alpha+\beta,\beta,\alpha),$$
which satisfies the right bracket condition. The sequence in \cite{Gor2} used in the
${\tt D}_4$-case satisfies the right bracket condition.
\end{exam}

\begin{prop}\label{ordering1} \it
Let  $\Psi$ be the height weighted function as in Example~\ref{rootorder}.
\begin{itemize}
\item[{\it i)}] Let ``\/$>$'' be either the lexicographic order or the $\Psi$-weighted lexicographic order. If $S$ satisfies the
left bracket condition, then the $(\mathbb N^N,>,S)$-{\it filtration} on $\mathcal{U}(\mathfrak n^-)$ is {quasi-commutative}.
\item[{\it ii)}] Let ``\/$>$'' be either the right lexicographic order or the $\Psi$-weighted right lexicographic order. If $S$ satisfies the
right bracket condition, then the $(\mathbb N^N,>,S)$-{\it filtration} on $\mathcal{U}(\mathfrak n^-)$ is {quasi-commutative}.
\end{itemize}
\end{prop}
\begin{proof} We only prove {\it i)}, the proof of {\it ii)} is analogous.
Let $m$ be an arbitrary monomial of elements in $S$, we have to show that $m={\bf f}^{{\bf exp}(m)}+$ a linear combination of ordered monomials
which are strictly smaller than ${\bf f}^{{\bf exp}(m)}$. The proof is by induction on the total degree ($=$ the number of factors) of $m$.

For a monomial of degree one there is nothing to show.
Assume $i<j$, then either $f_{\beta_j}f_{\beta_i}=f_{\beta_i} f_{\beta_j}$ or $f_{\beta_j}f_{\beta_i}=f_{\beta_i} f_{\beta_j} + a_k f_{\beta_k} $,
so $f_{\beta_j}f_{\beta_i}$ is equal to the ordered product $f_{\beta_i} f_{\beta_j}$ plus possibly a term which is of the same weight (and hence same height)
as $f_{\beta_i}f_{\beta_j}$ but strictly smaller in the lexicographic ordering. In addition, the term is of strictly less total degree.

Suppose that $m$ is of total degree $k+1$ and by induction we assume
that any monomial $m'$ of total degree at most $k$ can be written as a linear combination $m'={\bf f}^{{\bf exp}(m')}$+ a sum of
ordered monomials which are strictly smaller than ${\bf f}^{{\bf exp}(m')}$ and which are of total degree at most $k-1$.
So write $m=f_{\beta_j} m'$ for some $j$, where $m'$ is of total degree $k$. We can rewrite $m'$ by induction, so
$$
m=f_{\beta_j} m'=f_{\beta_j} \mathbf f^{\mathbf{exp}(m')} +\sum_\ell f_{\beta_j}m_\ell.
$$
The $m_\ell$ are ordered monomials, strictly smaller than $\mathbf f^{{\bf exp}(m')}$, and of total degree at most $k-1$.
By construction we have ${\bf exp}(m)={\bf exp}(f_{\beta_j} f^{{\bf exp}(m')})={\bf exp}(m')+{\bf e}_j$. Similarly, we
have ${\bf exp}(f_{\beta_j} m_\ell)={\bf exp}(m_\ell)+{\bf e}_j$. It follows that (all terms have the same weight and hence only the lexicographic order is relevant):
$$
{\bf f}^{{\bf exp}(m)}={\bf f}^{{\bf exp}(m')+{{\bf e}_j}}>_{lex}{\bf f}^{{\bf exp}(m_\ell)+{{\bf e}_j}}={\bf f}^{{\bf exp}(f_{\beta_j} m_\ell)}.
$$
Since the monomials $f_{{\beta_j}}m_\ell$ have at most $k$ factors, we can apply induction.
By the above it follows that all ordered monomials occurring in this linear combination are strictly smaller than ${\bf f}^{{\bf exp}(m)}$
and are of degree at most $k$.

It remains to consider the monomial $f_{\beta_j}{\bf  f}^{{\bf exp}(m')}$. If the monomial is already an ordered monomial, then the proof is finished,
so suppose it is not. To simplify the notation, assume without loss of generality that the first entry in ${\bf exp}(m')$ is nonzero.
(Otherwise one has to replace in the following the index 1 by the smallest index $i$ such that the corresponding entry in ${\bf exp}(m')\not=0$.)
So we have
$$
f_{{\beta_j}} {{\bf f}^{{\bf exp}(m')}}=f_{{\beta_j}} f_{{\beta_1}}{{\bf f}^{{\bf exp}(m')-{\bf e}_1}}
=f_{\beta_1}f_{\beta_j} {{\bf f}^{{\bf exp}(m')-{\bf e}_1}}+c [f_{j},f_{{\beta_1}}] {\bf f}^{{\bf exp}(m')-{\bf e}_1}.
$$
By induction on the number of factors we know we can replace $f_{{\beta_j}} {\bf f}^{{\bf exp}(m')-{\bf e}_1}$ by
${\bf f}^{{\bf exp}(m')-{\bf e}_1+{\bf e}_j}$ plus a sum of
ordered monomials which are strictly smaller (in the lexicographic order because all terms have the same weight)
than ${\bf f}^{{\bf exp}(m')-{\bf e}_1+{\bf e}_j}$ and are of degree at most $k-1$.
Since multiplying an ordered monomial by $f_{{\beta_1}}$ gives again an ordered monomial and we just add ${\bf e}_1$ to the exponent,
we see that we can replace $f_{{\beta_1}}f_{{\beta_j}} {\bf f}^{{\bf exp}(m')-{\bf e}_1}$ by $ {\bf f}^{{\bf exp}(m)}$ plus a sum of
ordered monomials which are strictly smaller than ${\bf f}^{{\bf exp}(m)}$ and all are of degree at most $k$.

If $[f_{j},f_{1}] =0$, then the proof is finished. Otherwise we know $[f_{j},f_{1}] = a_\ell f_\ell$ for some $\ell >1$.
Again by induction on the number of factors we know that we can replace $f_{\ell}{\bf f}^{{\bf exp}(m')-{\bf e}_1}$ by
${\bf f}^{{\bf exp}(m')-{\bf e}_1+e_\ell}$ plus a sum of
ordered monomials which are strictly smaller than ${\bf f}^{{\bf exp}(m')-{\bf e}_1+{\bf e}_\ell}$ and are of total degree at most $k-1$. Since $\ell>1$ we get
for the lexicographic order (which is the only relevant one here because ${\bf f}^{{\bf exp}(m)}$ and ${\bf f}^{{\bf exp}(m')-{\bf e}_1+{\bf e}_\ell}$ have the same weight):
$$
{\bf f}^{{\bf exp}(m)}={\bf f}^{{\bf exp}(m')+{\bf e}_j}>_{lex}{\bf f}^{{\bf exp}(m')}>_{lex}{\bf f}^{{\bf exp}(m')-{\bf e}_1+{\bf e}_\ell}.
$$
It follows: $m={\bf f}^{{\bf exp}(m)}$ plus a sum of
ordered monomials which are strictly smaller than ${\bf f}^{{\bf exp}(m)}$ and are of total degree at most $\vert {\bf exp}(m)\vert -1$.
\end{proof}

\printindex

\end{document}